\documentclass{jmlr}

\title[Optimal Epoch Stochastic Gradient Descent Ascent]{Optimal Epoch Stochastic Gradient Descent Ascent Methods for Min-Max Optimization}

\author{\Name{Yan Yan}$^1$ \Email{yan-yan-2@uiowa.edu}\\
\Name{Yi Xu}$^2$ \Email{statxy@gmail.com}\\
\Name{Qihang Lin}$^3$ \Email{qihang-lin@uiowa.edu}\\
\Name{Wei Liu}$^4$ \Email{wl2223@columbia.edu}\\
\Name{Tianbao Yang}$^1$ \Email{tianbao-yang@uiowa.edu} \\
\addr $^1$Department of Computer Science, University of Iowa \\
\addr $^2$DAMO Academy, Alibaba Group \\
\addr $^3$Department of Management Sciences, University of Iowa \\
\addr $^4$Tencent AI Lab
}

\usepackage[utf8]{inputenc} 
\usepackage[T1]{fontenc}    
\usepackage{hyperref}       
\usepackage{url}            
\usepackage{booktabs}       
\usepackage{amsfonts}       
\usepackage{nicefrac}       
\usepackage{microtype}      
\usepackage{amsmath,amssymb,mathrsfs,amsfonts,dsfont, tikz}
\usepackage{algorithm,algorithmic}
\usepackage{multirow}
\usepackage{multicol}

\usepackage{diagbox}
\usepackage{amsmath,amssymb}
\usepackage{amsmath,amsfonts,bm}
\usepackage{mathtools}

\usepackage{hyperref}
\hypersetup{final}

\newtheorem{thm}{Theorem}

\newtheorem{lem}{Lemma}

\newtheorem{defi}{Definition}

\newtheorem{assumption}{Assumption}

\usepackage{enumitem}
\usepackage{textcase,booktabs}

\def\calB{\mathcal{B}}

\def\xbar{{\bar{x}}}
\def\ybar{{\bar{y}}}

\def\xtilde{{\tilde{x}}}
\def\ytilde{{\tilde{y}}}

\def\E{\mathrm{E}}

\def\LHS{{\text{LHS}}}
\def\RHS{{\text{RHS}}}

\def\tildedelta{\tilde{\delta}}

\def\Gap{\text{Gap}}
\def\hatGap{\widehat{\text{Gap}}}

\def\dist{{\it{dist}}}

\newcommand\yancomment[1]{}

\begin{document}

\maketitle

\begin{abstract}
Epoch gradient descent method (a.k.a. Epoch-GD) proposed by \cite{hazan-2011-beyond} was deemed a breakthrough for stochastic strongly convex minimization, which achieves the optimal convergence rate of $O(1/T)$ with $T$ iterative updates for the {\it objective gap}. However, its extension to solving stochastic min-max problems with strong convexity and strong concavity still remains open, and it is still unclear whether a fast rate of $O(1/T)$ for the {\it duality gap} is achievable for stochastic min-max optimization under strong convexity and strong concavity. Although some recent studies have proposed stochastic algorithms with fast convergence rates for min-max problems, they require additional assumptions about the problem, e.g., smoothness, bi-linear structure, etc. 
In this paper, we bridge this gap by providing a sharp analysis of epoch-wise stochastic gradient descent ascent method (referred to as Epoch-GDA) for solving strongly convex strongly concave (SCSC) min-max problems, without imposing any additional assumption about smoothness or the function's structure. To the best of our knowledge, our result is the first one that shows Epoch-GDA can achieve the optimal rate of $O(1/T)$ for the duality gap of general SCSC min-max problems. 
We emphasize that such generalization of Epoch-GD for strongly convex minimization problems to Epoch-GDA for SCSC min-max problems is non-trivial and requires novel technical analysis. Moreover, we notice that the key lemma can also be  used for proving the convergence of Epoch-GDA for weakly-convex strongly-concave min-max problems, leading to a nearly optimal complexity without resorting to smoothness or other structural conditions. 
\end{abstract}

\section{Introduction}

In this paper, we consider {stochastic algorithms} for solving the following min-max saddle-point problem with a general objective function $f$ {\it without smoothness or any other special structure}:
\begin{align}\label{eq:objective}
\min_{x \in X} \max_{y \in Y} f(x, y)  ,
\end{align}
where $X \subseteq \mathbb R^d$ and $Y \subseteq \mathbb R^n$ are closed convex sets and $f: X \times Y \rightarrow \mathbb R$ is continuous.
It is of great interest to find a saddle-point solution to the above problem, which is defined as $(x^*, y^*)$  such that
$
f(x^*, y) \leq f(x^*, y^*) \leq f(x, y^*), \forall x\in X, y\in Y.
$
Problem (\ref{eq:objective}) covers a number of applications in machine learning, including distributionally robust optimization (DRO)~\cite{namkoongnips2017variance,DBLP:conf/nips/NamkoongD16}, learning with non-decomposable loss functions~\cite{fastAUC18,fanNIPS2017_6653,ying2016stochastic,liu2019stochastic}, and generative adversarial networks~\cite{Goodfellow:2014:GAN:2969033.2969125,pmlr-v70-arjovsky17a}.

In this work, we focus on two classes of the min-max problems: (i) strongly-convex strongly-concave (SCSC) problem where $f$ is strongly convex in terms of $x$ for any $y\in Y$ and is strongly concave in terms of $y$ for any $x\in X$; (ii) weakly-convex strongly-concave (WCSC) problem, where there exists $\rho>0$ such that $f(x, y) + \frac{\rho}{2}\|x\|^2$ is strongly convex in terms of $x$ for any $y\in Y$ and is strongly concave in terms of $y$ for any $x\in X$. Both classes have applications in machine learning~\cite{yan2019stochastic,hassan18nonconvexmm}.

Although stochastic algorithms for convex-concave min-max problems have been studied extensively in the literature, 
their research is still far behind its counterpart for stochastic convex minimization problems.   Below, we highlight some of these gaps to motivate the present work. 
For the sake of presentation, we first introduce some terminologies. 
The duality gap at $(x, y)$ is defined as 
$\Gap(x, y) := f(x, \hat y(x)) - f(\hat x(y), y)$,
where $\hat x(y) := \arg\min_{x' \in X} f(x', y)$ and $\hat y(x) := \arg\max_{y' \in Y} f(x, y')$.
If we denote by  $P(x) := \max_{y' \in Y} f(x, y')$, then $P(x) - P(x^*)$ is the primal objective gap, where $x^*= \arg\min_{x\in X}P(x)$.
 
 When $f$ is convex in $x$ and concave in $y$, many studies have designed and analyzed stochastic primal-dual algorithms for solving the min-max problems under different conditions of the problem (see references in next section).
A standard result is provided by~\cite{Nemirovski:2009:RSA:1654243.1654247}, which proves that  primal-dual SGD suffers from a convergence rate of $O(1/\sqrt{T})$ for the duality gap without imposing any additional assumptions about the objective function. This is analogous to that for stochastic convex minimization~\cite{Nemirovski:2009:RSA:1654243.1654247}. 
However, the research of stochastic algorithms for SCSC problems lacks behind that for strongly convex minimization problems. A well-known result for stochastic strongly convex minimization is given by \cite{hazan-2011-beyond}, which presents the first fast convergence rate $O(1/T)$ for stochastic strongly convex minimization by the Epoch-GD algorithm, which runs standard SGD in an epoch-wise manner by decreasing the step size geometrically.  However, a fast rate of $O(1/T)$ for the {\it duality gap} of a stochastic algorithm {\bf is still unknown for general SCSC problems}. 
We notice that there are extensive studies  about stochastic algorithms with faster convergence rates than $O(1/\sqrt{T})$ for solving convex-concave min-max problems~~\cite{zhang2017stochastic,tan2018stochastic,citeulike:11703902,du2018linear,dang2014randomized,chen2014optimal,DBLP:conf/nips/PalaniappanB16,hsieh2019convergence,yan2019stochastic,hien2017inexact,zhao2019optimal}.  However, these works usually require additional assumptions about the objective function (e.g., smoothness, bilinear structure) or only prove the convergence in weaker measures (e.g., the primal objective gap, the distance of a solution to the saddle point).

We aim  to bridge this gap by presenting the first optimal rate $O(1/T)$ of the duality gap for solving general SCSC problems. In particular, we propose an epoch-wise stochastic gradient descent ascent (Epoch-GDA) algorithm - a primal-dual variant of Epoch-GD that runs stochastic gradient descent update for the primal variable and stochastic gradient ascent update for the dual variable for solving~(\ref{eq:objective}). Although the algorithmic  generalization is straightforward, the proof of convergence in terms of the duality gap for Epoch-GDA is not straightforward  at all. We note that the key difference in the analysis of Epoch-GDA is that to upper bound the duality gap of a solution $(\bar x, \bar y)$ we need to deal with the distance of an initial solution $(x_0, y_0)$ to the reference solutions $(\hat x(\bar y), \hat y(\bar x))$, where $\hat x(\bar y)=\arg\min_{x'\in X}f(x', \bar y)$ and $\hat y(\bar x) = \arg\max_{y'\in Y}f(\bar x, y')$ depend on $\bar y$ and $\bar x$, respectively.
In contrast, in the  analysis of the objective gap for Epoch-GD, one only needs to deal with the distance from an initial solution $x_0$ to the optimal solution $x^*$, i.e., $\|x_0 - x^*\|_2^2$, which by strong convexity can easily connects to the objective gap $P(x_0) - P(x^*)$, leading to the telescoping sum on the objective gap. Towards addressing the challenge caused by dealing with the duality gap, we present a key lemma that connects the distance measure $\|x_0 -\hat x(\bar y) \|_2^2 +\|y_0 -\hat y(\bar x) \|_2^2 $ to the duality gap of $(x_0, y_0)$ and $(\bar x, \bar y)$. In addition, since we use the same technique as Epoch-GD for handling the variance of stochastic gradient by projecting onto a bounded ball with shrinking radius, we have to carefully prove that such restriction does not affect the duality gap for the original problem, which also needs to deal with bounding  $\|x_0 -\hat x(\bar y)\|_2^2$ and $\|y_0 -\hat y(\bar x)\|_2^2$. 

Moreover, we notice that the aforementioned  key lemma and the telescoping technique  based on the duality gap  can also be used for proving the convergence of Epoch-GDA for {\bf finding an approximate stationary solution of general WCSC problems}. The algorithmic framework  is similar to that proposed by~\cite{hassan18nonconvexmm}, i.e., by solving SCSC problems successively, but with a subtle difference in handling the dual variable. In particular,  we do not need additional condition on the structure of the objective function and extra care for dealing with the dual variable for restart as done in~\cite{hassan18nonconvexmm}. This key difference is caused by our sharper analysis, i.e., we use the telescoping sum based on the duality gap instead of the primal objective gap as in~\cite{hassan18nonconvexmm}.  As a result,  our algorithm and analysis lead to a nearly optimal complexity  for solving WCSC problems without the smoothness assumption on the objective~\cite{arjevani2019lower}~\footnote{Although~\cite{arjevani2019lower} only concerns the lower bound of finding a stationary point of smooth non-convex problems $\min_{x}f(x)$ through stochastic first-order oracle, it is a special case of the WCSC problem. }. 
Finally, we summarize our results and the comparison with existing results in Table~\ref{tab:1}. 
\begin{table*}[t]
	\caption{Summary of complexity results of this work and previous works for finding an $\epsilon$-duality-gap  solution for SCSC  or an $\epsilon$-stationary solution for WCSC min-max problems. We focus on comparison of existing results without assuming smoothness of the objective function. Restriction means whether an additional condition about the objective function's structure is imposed. 
}\label{tab:1}
	\centering
	\label{tab:2}
	\resizebox{\textwidth}{!} 
	{
	\scalebox{1}{\begin{tabular}{l|l|l|l|l}
			\toprule
			Setting &  Works  &  Restriction &  Convergence  & Complexity\\
			\midrule
			&\cite{Nemirovski:2009:RSA:1654243.1654247}&No&Duality Gap& $O\left(1/\epsilon^2\right) $\\
			SCSC&\cite{yan2019stochastic} & Yes & Primal Gap & $O\left(1/\epsilon + n \log(1/\epsilon)\right)$ \\
			&{\bf{This paper }}& {\bf{ No }}& {\bf{Duality Gap}} & {\bm{ $O\left(1/\epsilon\right)$}} \\
			\midrule
			&\cite{hassan18nonconvexmm}&  No  & Nearly Stationary & {$ \widetilde O\left(1/\epsilon^6\right)$}\\\ 
			WCSC &\cite{hassan18nonconvexmm}&  Yes  & Nearly Stationary & {$ \widetilde O\left(1/\epsilon^4+ n/\epsilon^2\right)$}\\ 
   		& {\bf{This paper}} & {\bf{No}} & {\bf{Nearly Stationary}} & {\bm{$\widetilde O\left( 1/\epsilon^4 \right)$}}\\ 
			\bottomrule
	\end{tabular}}
	}
\end{table*}
\setlength{\abovedisplayskip}{3pt}
\setlength{\belowdisplayskip}{3pt}

\section{Related Work }

Below, we provide an overview of related results in this area and the review is not necessarily exhaustive. In addition, we focus on the stochastic algorithms, and leave deterministic algorithms~\cite{Chambolle:2011:FPA:1968993.1969036,nesterovexces,DBLP:journals/ml/YangMJZ15,gidel2016frank,nouiehed2019solving,hong2016decomposing,hajinezhad2019perturbed,hong2018gradient,lu2019hybrid,hamedani2018primal} out of our discussion. 

\cite{Nemirovski:2009:RSA:1654243.1654247} is one of the early works that studies  stochastic primal-dual gradient methods for convex-concave min-max problems, which establishes a convergence rate of $O(1/\sqrt{T})$ for the duality gap of general convex-concave problems.
Following this work, many studies have tried to improve the algorithm and the analysis for a certain class of problems by exploring the smoothness condition of some component functions~\cite{juditsky2011solving,zhao2019optimal,hsieh2019convergence} or  bilinear structure of the objective function~\cite{chen2014optimal,dang2014randomized}. For example,  \cite{zhao2019optimal} considers a family of min-max problems whose objective is $f(x) + g(x) + \phi(x, y) - J(y)$, where the smoothness condition is imposed on $f$ and $\phi$ and strong convexity is imposed on $f$ if necessary, and  establishes optimal or nearly optimal complexity of a stochastic primal-dual hybrid algorithm. Although the dependence on each problem parameter of interest is made (nearly) optimal, the worst case complexity is still $O(1/\sqrt{T})$. \cite{hsieh2019convergence} considers single-call stochastic extra-gradient and establishes $O(1/T)$ rate for smooth and strongly monotone variational inequalities in terms of  the square distance from the returned solution to the saddle point.  The present work is complementary to these developments by making no assumption on smoothness or the structure of the objective but considers strong (weak) convexity and strong concavity of the objective function. It has applications in robust learning with non-smooth loss functions~\cite{yan2019stochastic,hassan18nonconvexmm}.


In the machine learning community, many works have considered stochastic primal-dual algorithms for solving regularized loss minimization problems, whose min-max formulation usually exhibits bi-linear structure~\cite{zhang2017stochastic,citeulike:11703902,wang2017exploiting,du2018linear,DBLP:conf/nips/PalaniappanB16}. For example, \cite{zhang2017stochastic} designs a stochastic primal-dual coordinate (SPDC) method for SCSC problems with bilinear structure, which enjoys a linear convergence for the duality gap.
Similarly, in~\cite{DBLP:journals/corr/YuLY15,tan2018stochastic}, different variants of SPDC are proposed and analyzed for problems with the bilinear structure.
\cite{DBLP:conf/nips/PalaniappanB16} proposes  stochastic variance reduction methods for a family of saddle-point problems with special structure that yields a linear convergence rate.
An exception that makes no smoothness assumption and imposes no bilinear structure is a recent work \cite{yan2019stochastic}. It considers a family of functions $f(x, y) = y^{\top}\ell(x) - \phi^*(y) + g(x)$ and proposes  a stochastic primal-dual algorithm similar to Epoch-GDA. The key difference is that  \cite{yan2019stochastic} designs a particular scheme that computes a restarting dual solution based on $\nabla\phi(\ell(\bar x))$, where $\bar x$ is a restarting primal solution in order to derive a fast rate of $O(1/T)$ under strong convexity and strong concavity. Additionally, their fast rate $O(1/T)$ is in terms of the primal objective gap, which is weaker than our convergence result in terms of the duality gap.

There is also increasing interest in stochastic primal-dual algorithms for solving WCSC min-max problems. 
To the best of our knowledge, \cite{hassan18nonconvexmm} is probably the first work that comprehensively studies stochastic algorithms for solving WCSC min-max problems.  
To find a nearly $\epsilon$-stationary point, their algorithms  suffer from  an $O(1/\epsilon^6)$ iteration complexity without strong concavity and an $O(1/\epsilon^4+n/\epsilon^2)$ complexity with strong concavity and a special structure of the objective function that is similar to that imposed in \cite{yan2019stochastic}.
Some recent works are trying to improve the complexity for solving WCSC min-max problems by exploring other conditions (e.g., smoothness)~\cite{DBLP:journals/corr/abs-1906-00331,arXiv:2001.03724}. For example,  \cite{DBLP:journals/corr/abs-1906-00331} establishes an $O(1/\epsilon^4)$ complexity for a single-loop stochastic gradient descent ascent method. However, their analysis requires the smoothness condition and their algorithm needs to use  a large mini-batch size in the order $O(1/\epsilon^2)$. In contrast, we impose neither assumption about smoothness nor special structure of the objective function. The complexity of our algorithm is $\widetilde O(1/\epsilon^4)$ for finding a nearly $\epsilon$-stationary point, which is the state of the art result for the considered non-smooth WCSC problem. 

\section{Preliminaries}

This section provides some notations and assumptions used in the paper.
We let $\|\cdot\|$ denote the Euclidean norm of a vector. 
Given a function $f: \mathbb R^d \rightarrow \mathbb R$, we denote the Fr\'echet subgradients and limiting Fr\'echet gradients by $\hat{\partial} f$ and $\partial f$ , respectively, i.e., at $x$, 
$
\hat{\partial} f(x) = \{ v \in \mathbb R^d : \lim_{x \rightarrow x'} \inf \frac{ f(x) - f(x') - v^{\top} ( x - x' ) }{ \| x - x' \| } \geq 0 \}  ,
$ 
and
$
\partial f(x) = \{ v_k \in \mathbb R^d : \exists x_{k} \stackrel{f}{\rightarrow} x, v_{k} \in \hat{\partial} f(x_{k}), v_{k} \rightarrow v, v \in \hat \partial f(x) \}.
$
Here $x_{k} \stackrel{f}{\rightarrow} x$ represents $x_{k} \rightarrow x$ with $f(x_{k}) \rightarrow f(x)$.
A function $f(x)$ is $\mu$-strongly convex on $X$ if for any $x, x' \in X$,
$
\partial f(x')^\top (x - x') + \frac{\mu}{2} \| x - x' \|^2 \leq f(x) - f(x')  .
$
A function $f(x)$ is $\rho$-weakly convex on $X$ for any $x, x' \in X$
$
\partial f(x')^\top (x - x') - \frac{\rho}{2} \| x - x' \|^2 \leq f(x) - f(x')  .
$
Let $\mathcal G_x = \partial_x f(x, y; \xi)$ denote a stochastic subgradient of $f$ at $x$ given $y$,
where $\xi$ is used to denote the random variable. 
Similarly, let $\mathcal G_y = \partial_y f(x, y; \xi)$ denote a stochastic sugradient of $f$ at $y$ given $x$.
Let $\Pi_\Omega[\cdot]$ denote the projection onto the set $\Omega$, and let $\mathcal B(x, R)$ denote an Euclidean ball centered at $x$ with a radius $R$. 
Denote by $\dist(x, X)$ the distance between $x$ and the set $X$, i.e., $\dist(x, X) = \min_{v \in X} \| x - v \|$.
Let $\tilde O(\cdot)$ hide some logarithmic factors.


For a WCSC min-max problem, it is generally a hard problem to find a saddle point. Hence, we use {\it nearly $\epsilon$-stationarity} as the measure of convergence for solving WCSC problems~\cite{hassan18nonconvexmm}, which is defined as follows.
\begin{defi}
A solution $x$ is a nearly $\epsilon$-stationary point of $\min_x\psi(x)$ if there exist $z$ and a constant $c > 0$ such that $\| z - x \| \leq c \epsilon$ and $\dist(0, \partial \psi(z) ) \leq \epsilon$.
\end{defi}
\noindent
For a $\rho$-weakly convex function $\psi(x)$, let $z = \arg\min_{x \in \mathbb R^d} \psi(x) + \frac{\gamma}{2} \| x - \tilde x \|^2$ where $\gamma > \rho$ and $\tilde x \in \mathbb R^d$ is a reference point.
Due to the strong convexity of the above problem, $z$ is unique and $0 \in \partial \psi(z) + \gamma ( z - \tilde x )$,
which results in $\gamma (\tilde x - z) \in \partial \psi(z)$,
so that $\dist(0, \partial \psi(z)) \leq \gamma \| \tilde x - z \|$.
We can find a nearly $\epsilon$-stationary point $\tilde x$ as long as $\gamma \| \tilde x - z \| \leq \epsilon$. 

Before ending this section, we present some assumptions that will be imposed in our analysis.
\begin{assumption}\label{ass:1}
$X$ and $Y$ are closed convex sets.
There exist initial solutions $x_0\in X, y_0\in Y$ and $\epsilon_0>0$ such that $\Gap(x_0, y_0)\leq \epsilon_0$. 
\end{assumption}

\begin{assumption}\label{ass:2}
(1) $f(x, y)$ is $\mu$-strongly convex in $x$ for any $y \in Y$ and $\lambda$-strongly concave in $y$ for any $x \in X$.
(2) There exist $B_1, B_2>0$ such that  $\E[ \exp( \frac{ \| \mathcal G_x \|^2 }{B_1^2} ) ] \leq \exp(1)$ and
$\E [ \exp( \frac{ \| \mathcal G_y \|^2 }{B_2^2} ) ] \leq \exp(1)$.
\end{assumption}

\begin{assumption}\label{ass:3}
(1) $f(x, y)$ is $\rho$-weakly convex in $x$ for any $y \in Y$ and is  $\lambda$-strongly concave in $y$ for any $x \in X$.
(2) $\E[ \| \mathcal G_x \|^2 ] \leq M_1^2$ and $\E [ \| \mathcal G_y \|^2 ] \leq M_2^2$.
\end{assumption}
{\bf Remark:} When $f(x, y)$ is smooth in $x$ and $y$, the second condition in the above assumption can be replaced by the bounded variance condition. 

\section{Main Results }

\subsection{Strongly-Convex Strongly-Concave Min-Max Problems}
\label{subsection:scsc}

In this subsection, we present the main result for solving SCSC problems. The proposed Epoch-GDA algorithm for SCSC min-max problems is shown in Algorithm~\ref{alg:RSPD_scsc}.
As illustrated, our algorithm consists of a series of epochs.
In each epoch (Line \ref{alg1_modified:line:pd_sa_inner_loop_start} to \ref{alg1_modified:line:pd_sa_inner_loop_end}), standard primal-dual updates are performed.
After an epoch ends, in Line \ref{line_modified:pd_sa_update_x_initial}, the solutions $\bar x_k$ and $\bar y_k$ averaged over the epoch are returned as the initialization for the next epoch.
In Line \ref{alg1_modified:line:pd_update_R}, step sizes $\eta_{x, k+1}$ and $\eta_{y, k+1}$, the radius $R_{k+1}$ and the number of iterations $T_{k+1}$ are also adjusted for the next epoch.
The ball constraints  $\mathcal B(x_0^k, R_k)$ and $\mathcal B(y_0^k, R_k)$ at each iteration are used for the convergence analysis in high probability  as  in \cite{hazan-2011-beyond,hazan2014beyond}. It is clear that Epoch-GDA can be considered as a primal-dual variant of Epoch-GD~\cite{hazan-2011-beyond,hazan2014beyond}. 


The following theorem shows that the iteration complexity of Algorithm \ref{alg:RSPD_scsc} to achieve an $\epsilon$-duality gap  for a general SCSC problem (\ref{eq:objective}) is $O(1/\epsilon)$.

\setlength{\textfloatsep}{5pt}

\begin{algorithm}[t]
\caption{Epoch-GDA for SCSC Min-Max Problems}
\label{alg:RSPD_scsc}
\begin{algorithmic}[1]
\STATE Init.: $x_{0}^1 = x_0 \in X$,
              $y_{0}^1 = y_0 \in Y$, $ \eta^1_x, \eta^1_y, R_1, T_1$
\FOR{$k = 1, 2, ..., K$}

\FOR{$t = 0, 1, 2, ..., T_k - 1$}
\label{alg1_modified:line:pd_sa_inner_loop_start}

  \STATE Compute stochastic gradients $\mathcal G^k_{x, t} = \partial_x f(x_t^k, y_t^k; \xi_t^k)$ and $\mathcal G^k_{y, t} = \partial_y f(x_t^k, y_t^k; \xi_t^k)$.
  \STATE $x_{t+1}^k = \Pi_{X \cap \mathcal B(x_0^k, R_k)} (x_{t}^k - \eta_x^k\mathcal G^k_{x, t}) $
  \label{alg1_modified:line:pd_sa_update_x}
  \STATE $y_{t+1}^k = \Pi_{Y \cap \mathcal B(y_0^k, R_k)} (y_t^k + \eta_y^k \mathcal G^k_{y,t}) $
  \label{alg1_modified:line:pd_sa_update_y}
\ENDFOR
\label{alg1_modified:line:pd_sa_inner_loop_end}

\STATE $x_{0}^{k+1} = \xbar_{k} = \frac{1}{T_k} \sum_{t=0}^{T_k-1} x_{t}^k$ 
\label{line_modified:pd_sa_update_x_initial}, 
       $y_{0}^{k+1} = \ybar_{k} = \frac{1}{T_k} \sum_{t=0}^{T_k-1} y_{t}^k$
\label{line_modified:pd_sa_update_y_initial}
\STATE $\eta_{x}^{k+1} = \frac{\eta_x^k }{2}$,
       $\eta_{y}^{k+1} = \frac{\eta_y^k }{2}$, 
       $R_{k+1} = R_k / \sqrt{2}$,
       $T_{k+1} = 2 T_k$.
\label{alg1_modified:line:pd_update_R}
\ENDFOR
\STATE Return $(\bar x_K, \bar y_K).$
\end{algorithmic}
\end{algorithm}

\begin{thm}\label{thm:tot_iter_scsc}
Suppose Assumption \ref{ass:1} and Assumption~\ref{ass:2} hold and let $\delta\in(0,1)$ be a failing probability and $\epsilon\in(0,1)$ be the target accuracy level for the duality gap.
Let   $K = \lceil \log(\frac{\epsilon_0}{\epsilon}) \rceil$ 
and
$\tildedelta = \delta/K $, and the initial parameters are set by $R_1\geq2 \sqrt{ \frac{ 2 \epsilon_0 }{ \min\{ \mu, \lambda \} } } $,
$\eta_x^1 = \frac{ \min\{ \mu, \lambda \} R_1^2 }{ 40 ( 5 + 3 \log(1/\tilde \delta) ) B_1^2 }$,
$\eta_y^1 = \frac{ \min\{ \mu, \lambda \} R_1^2 }{ 40 ( 5 + 3 \log(1/\tilde \delta) ) B_2^2 }$ and
$$T_1 \geq \frac{ \max \Big\{ 320^2 (B_1 + B_2)^2 3 \log( 1 / \tilde \delta ), 3200 ( 5 + 3 \log(1/\tilde \delta) ) \max\{ B_1^2, B_2^2 \} \Big\} }{ \min\{\mu, \lambda\}^2 R_1^2 }.$$
Then the total number of iterations of Algorithm \ref{alg:RSPD_scsc} to achieve an $\epsilon$-duality gap, i.e., $\Gap(\bar x_K, \bar y_K) \leq \epsilon$, with probability $1 - \delta$ is
$$
T_{tot} = \frac{ \max \Big\{ 320^2 (B_1 + B_2)^2 3 \log(\frac{1}{\tildedelta}), 3200 ( 5 + 3 \log(1/\tilde \delta) ) \max\{ B_1^2, B_2^2 \} \Big\} }{ 4 \min\{ \mu, \lambda \} \epsilon } . 
$$
\end{thm}

\begin{remark}
To the best of our knowledge, this is the first study that achieves a fast rate of $O(1/T)$ for the duality gap of a general SCSC min-max problem without any special structure assumption or smoothness of the objective function and an additional computational cost.
In contrast, even if the algorithm in \cite{yan2019stochastic} attains the $O(1/T)$ rate of convergence, it i) only guarantees the convergence of the primal objective gap, rather than the duality gap, ii) additionally requires a special structure of the objective function,
and iii) needs an extra $O(n)$ computational cost of the deterministic update at each outer loop to handle the maximization over $y$.
In contrast, Algorithm \ref{alg:RSPD_scsc} has stronger theoretical results with less restrictions of the problem structures and computational cost.
\end{remark}

\begin{remark}
{
A lower bound of $O(1/T)$ for stochastic strongly convex minimization problems has been proven in \cite{agarwal2009information,hazan2014beyond}.
Due to $\Gap(x, y) \geq P(x) - P(x^*)$, bounding the duality gap is more difficult than bounding the primal gap.
This means that our convergence rate matches the lower bound and is therefore the best possible convergence rate without adding more assumptions.
}
\end{remark}

\subsection{Weakly-Convex Strongly-Concave Problems}
\label{subsection:wcsc}

\begin{algorithm}[t]
\caption{Epoch-GDA for WCSC Min-Max Problems}
\label{alg:wcc}
\begin{algorithmic}[1]
\STATE Init.: $x_{0}^{1} = x_0 \in X$, 
              $y_{0}^{1} = y_0 \in Y$,
              $\gamma = 2 \rho$.

\FOR{$k = 1, 2, ..., K$}

  \STATE Set $T_k = \frac{106(k+1)}{3}$,
             $\eta_x^k = \frac{ 4 }{\rho ( k + 1 ) }$,
             $\eta_y^k = \frac{ 2 }{\lambda ( k + 1 ) }$.
  
\FOR{$t = 1, 2, ..., T_k$}
\label{alg:wcc:line:inner0}

  \STATE Compute $\mathcal G^k_{x, t} = \partial_x f (x_t^k, y_t^k; \xi_t^k)$ and
             $\mathcal G^k_{y, t} = \partial_y f (x_t^k, y_t^k; \xi_t^k)$.
             
  \STATE $x_{t+1}^{k} = \arg\min_{x \in X} x^\top \mathcal G^{k}_{x, t} + \frac{1}{2\eta_x^k} \| x - x_{t}^k \|^2 + \frac{\gamma}{2} \| x - x_0^k \|^2 $
  \label{alg:wcc:line:update_x}

  \STATE $y_{t+1}^{k} = \arg\min_{y \in Y} - y^\top \mathcal G^{k}_{y,t} + \frac{1}{2\eta_y^k} \| y - y_{t}^k \|^2 $

\ENDFOR
\label{alg:wcc:line:inner1}

\STATE $x_{0}^{k+1} = \xbar_{k} = \frac{1}{T} \sum_{t=0}^{T-1} x_{t}^{k}$,  $y_{0}^{k+1} = \ybar_{k} = \frac{1}{T} \sum_{n=0}^{T-1} y_{t}^{k}$
\label{alg:wcc:line:average_solution}

\ENDFOR

\STATE Return $x_0^\tau$ by $\tau$ randomly sampled from $\{ 1, ..., K \}$.
\end{algorithmic}
\end{algorithm}
In this subsection, we present the convergence results for solving WCSC problems, where the objective function $f(x, y)$ in (\ref{eq:objective}) is $\rho$-weakly convex in $x$ and $\lambda$-strongly concave in $y$.
The proposed Epoch-GDA algorithm for WCSC min-max problems is summarized in Algorithm \ref{alg:wcc}. As our Algorithm \ref{alg:RSPD_scsc}, Algorithm \ref{alg:wcc} consists of a number of epochs.
As shown in Line \ref{alg:wcc:line:inner0} to Line \ref{alg:wcc:line:inner1}, each epoch performs primal-dual updates on $x$ and $y$.
When updating $x$ at the $k$-th stage, an additional regularizer $\frac{\gamma}{2} \| x - x_0^k \|^2$ is added, where the value  $\gamma = 2 \rho$.
The added term is used to handle the weak convexity condition.
After an epoch ends, average solutions of both $x$ and $y$ are restarted as the initial ones for the next epoch.
The step sizes for updating $x$ and $y$ are set to $O(1/(\rho k))$ and $O(1/(\lambda k))$ at the $k$-th epoch, respectively. 
If we define $\hat f_k(x, y) = f(x, y) + \frac{ \gamma }{ 2 } \| x - x_0^k \|^2$, we can see that  $\hat f_k(x, y)$ is $\rho$-strongly convex in $x$ and $\lambda$-strongly concave in $y$, since $f(x, y)$ is $\rho$-weakly convex and $\gamma = 2 \rho$.
Indeed, for each inner loop of Algorithm \ref{alg:wcc}, we actually work on the SCSC problem $\min_{x \in X} \max_{y \in Y} \hat f_k(x, y)$.

It is worth mentioning the key difference between our algorithm and the recently proposed stochastic algorithm PG-SMD~\cite{hassan18nonconvexmm} for WCSC problems with a  special structural  objective function.
PG-SMD also consists of two loops.
For each inner loop, it runs the same updates with the added regularizer on $x$ as Algorithm \ref{alg:wcc}.
It restarts $x$ by averaging the solutions over the inner loop, like our $\bar x_k$,
but restarts $y$ by taking the deterministic maximization of (\ref{eq:objective}) over $y$ given $\bar x_k$, leading to an additional $O(n)$ computational complexity per epoch.
In addition, PG-SMD sets $\eta_y^k = O( 1 / (\gamma \lambda^2 k ) )$.
Although Algorithm \ref{alg:wcc} shares similar updates to PG-SMD, our analysis yields stronger results under weaker assumptions --- the same iteration complexity $\tilde O(1/\epsilon^4)$ without deterministic updates for $y$ and special structure in the objective function.
This is due to our sharper analysis that makes use of  the telescoping sum based on the duality gap of $\hat f_k$ instead of the primal objective gap. 

The convergence result of Algorithm \ref{alg:wcc} that achieves a nearly $\epsilon$-stationary point with $\tilde O(1/\epsilon^4)$ iteration complexity is summarized below.
\begin{thm}
\label{thm:wcc_convergence}
Suppose Assumption \ref{ass:3} holds.
Algorithm \ref{alg:wcc} guarantees 
$\E [\dist (0, \partial P(\hat x^*_\tau))^2]
\leq \gamma^2 \E [ \| \hat x^*_\tau - x_0^\tau \|^2 ]
\leq \epsilon^2$
after $K = \max \left\{ \frac{1696 \gamma (\frac{2 M_1^2}{\rho} + \frac{M_2^2}{\lambda} ) }{ \epsilon^2 } \ln( \frac{1696 \gamma (\frac{2 M_1^2}{\rho} + \frac{M_2^2}{\lambda} ) }{ \epsilon^2 }), 
\frac{ 1376 \gamma \epsilon_0 }{5 \epsilon^2} \right\}$ epochs, where  $\tau$ is randomly sampled from $\{ 1, ..., K \}$ and $(\hat x^*_k, \hat y^*_k)$ is the saddle-point of $f_k(x, y)$.
The total number of iteration is $\sum_{k=1}^{K} T_k = \tilde O(\frac{1}{\epsilon^4})$.
\end{thm}

\begin{remark}
Theorem \ref{thm:wcc_convergence} shows that the iteration complexity of Algorithm \ref{alg:wcc} to attain an $\epsilon$-nearly stationary point is $\tilde O(1/\epsilon^4)$.
It improves the result of \cite{hassan18nonconvexmm} for WCSC problems in terms of two aspects.
First, \cite{hassan18nonconvexmm} requires a stronger condition on the structure of the objective function,
while our analysis simply assumes a general objective function $f(x, y)$.
Second, \cite{hassan18nonconvexmm} requires to solve the maximization over $y$ at each epoch,
which may introduce an $O(n)$ computational complexity for $y \in \mathbb R^n$~\footnote{Although the exact maximization over $y$ for restarting next epoch might be solved approximately, it still requires additional overhead.}.
In contrast, our algorithm restarts both the primal variable $x$ and dual variable $y$ at each epoch, which does not need an additional cost. 

Finally, we note that when $f(x, y)$ is smooth in $x$ and $y$, we can use stochastic Mirror Prox algorithm~\cite{juditsky2011solving} to replace the stochastic gradient descent ascent updates (Step 6 and Step 7) such that we can use a bounded variance assumption of the stochastic gradients instead of bounded second-order moments. It is a simple exercise to finish the proof by following our analysis of Theorem~\ref{thm:wcc_convergence}. 
\end{remark}

\section{Analysis}
In this section, we present the proof of Theorem~\ref{thm:tot_iter_scsc} and a proof sketch of Theorem~\ref{thm:wcc_convergence}.  As we mentioned at the introduction, the key challenge in the analysis of Epoch-GDA lies in handling the variable distance measure $\| \hat x(y_1) - x_0 \|^2 + \| \hat y(x_1) - y_0 \|^2$ for any $(x_0, y_0)\in X\times Y$ and $(x_1, y_1)\in X\times Y$ and its connection to the duality gaps, where $\hat x(y_1)=\arg\min_{x'\in X}f(x', y_1)$ and $\hat y(x_1) = \arg\max_{y'\in Y}f(x_1, y')$.  Hence, we first introduce the following key lemma that is useful in the analysis of Epoch-GDA for both SCSC and WCSC problems. It connects the variable distance measure $\| \hat x(y_1) - x_0 \|^2 + \|\hat y(x_1) - y_0 \|^2$ to the duality gaps at $(x_0, y_0)$ and $(x_1, y_1)$. 
\begin{lem}
\label{lemma:key_lemma}
Consider the following  $\mu$-strongly convex in $x$ and $\lambda$-strongly concave problem
$\min_{x \in \Omega_1} \max_{y \in \Omega_2} f(x, y)$. 
Let $(x^*, y^*)$ denote the saddle point solution to this problem.
Suppose we have two solutions $(x_0, y_0) \in \Omega_1 \times \Omega_2$ and $(x_1, y_1) \in \Omega_1 \times \Omega_2$.
Then the following relation between variable distance and duality gaps holds
\begin{align}
\label{eq:key_lemma}
\frac{\mu}{4} \| \hat x(y_1) - x_0 \|^2
+ \frac{\lambda}{4} \| \hat y(x_1) - y_0 \|^2
\leq &
\max_{y' \in \Omega_2} f(x_0, y') - \min_{x' \in \Omega_1} f(x', y_0)
\nonumber\\
+ &
\max_{y' \in \Omega_2} f(x_1, y') - \min_{x' \in \Omega_1} f(x', y_1)  .
\end{align}
\end{lem}

\subsection{Proof of Theorem~\ref{thm:tot_iter_scsc} for the SCSC setting}

The key idea is to first show the convergence of the duality gap with respect to the ball constraints $\calB(x_0^k, R_k)$ and $\calB(y_0^k, R_k)$ in an epoch (Lemma \ref{lemma:convergence_RSPDsc_per_stage}).
Then we investigate the condition to make $\hat x(\bar y_k) \in \calB(x_0^k, R_k)$ and $\hat y(\bar x_k) \in \calB(y_0^k, R_k)$ given the average solution $(\bar x_k, \bar y_k)$, which allows us to derive the duality gap $\Gap(\bar x_k, \bar y_k)$ for the original problem.
Finally, under such conditions, we show how the duality gap between two consecutive outer loops can be halved (Theorem~\ref{thm:connected_PD_gap_scsc}),
which implies the total iteration complexity (Theorem~\ref{thm:tot_iter_scsc}). Below, we omit superscript $k$ when it applies to all epochs.

\begin{lem}\label{lemma:convergence_RSPDsc_per_stage}
Suppose Assumption~\ref{ass:2} holds. 
Let Line \ref{alg1_modified:line:pd_sa_inner_loop_start} to \ref{alg1_modified:line:pd_sa_inner_loop_end} of Algorithm~\ref{alg:RSPD_scsc} run for $T$ iterations (omitting the $k$-index)
by fixed step sizes $\eta_x$ and $\eta_y$. 
Then with the probability at least $1 - \tildedelta$ where $0 < \tildedelta < 1$, 
for any $x \in X \cap \mathcal B(x_0, R)$ and $y \in Y \cap \mathcal B(y_0, R)$,  $\bar x = \sum_{t=0}^{T-1} x_{t}/T$,  $\bar y = \sum_{t=0}^{T-1} y_{t}/T$ satisfy
\begin{align}
\label{eq1:convergence_per_stage}
       f(\bar x, y) - f(x, \bar y) 
\leq & \frac{ \| x - x_{0} \|^{2} }{ \eta_x T } 
       + \frac{ \| y - y_{0} \|^{2} }{ \eta_y T } 
       + \frac{ \eta_x B_1^2 + \eta_y B_2^2  }{ 2 } ( 5 + 3 \log(1/\tilde \delta) )
     \nonumber\\
     & + \frac{ 4 ( B_1 + B_2 ) R \sqrt{ 3 \log(1/\tilde\delta) } }{ \sqrt{T} } . 
\end{align}
\end{lem}

\begin{remark}
Lemma \ref{lemma:convergence_RSPDsc_per_stage} is a standard analysis for an epoch of Algorithm \ref{alg:RSPD_scsc}.
The difficulty arises when attempting to plug $x$ and $y$ into (\ref{eq1:convergence_per_stage}).
In order to derive the duality gap on the LHS of (\ref{eq1:convergence_per_stage}), we have to plug in $x \leftarrow \hat x(\bar y)$ and $y \leftarrow \hat y(\bar x)$.
Nevertheless, it is unclear whether $\hat x(\bar y) \in \calB(x_0, R)$ and $\hat y(\bar x) \in \calB(y_0, R)$, which is the requirement for $x$ and $y$ to be plugged into (\ref{eq1:convergence_per_stage}).
In the following lemma, we investigate the condition to make $\hat x(\bar y) \in \calB(x_0, R)$ and $\hat y(\bar x) \in \calB(y_0, R)$ based on Lemma \ref{lemma:key_lemma}.
\end{remark}

\begin{lem}\label{lem:conditions_for_primaldual_gap}
Suppose Assumption~\ref{ass:2} holds. 
Let
$\hat{x}_R (y) := \arg\min_{x \in X \cap \mathcal B(x_0, R)} f(x, y)$
and
$\hat{y}_R (x) := \arg\max_{y \in Y \cap \mathcal B(y_0, R)} f(x, y)$.
Assume the initial duality gap 
$\Gap(x_0, y_0) \leq \epsilon_0$.
Let Lines \ref{alg1_modified:line:pd_sa_inner_loop_start} to \ref{alg1_modified:line:pd_sa_inner_loop_end} of Algorithm \ref{alg:RSPD_scsc} run $T$ iterations with
$\tilde \delta \in (0, 1)$,
$R \geq 2 \sqrt{ \frac{ 2 \epsilon_0 }{ \min\{ \mu, \lambda \} } }$,
$\eta_x = \frac{ \min\{ \mu, \lambda \} R^2}{ 40 ( 5 + 3 \log(1 / \tilde \delta) ) B_1^2 }$,
$\eta_y = \frac{ \min\{ \mu, \lambda \} R^2}{ 40 ( 5 + 3 \log(1 / \tilde \delta) ) B_2^2 }$ and
$$T \geq \frac{ \max \Big\{ 320^2 (B_1 + B_2)^2 3 \log(\frac{1}{\tildedelta}), 3200 ( 5 + 3 \log(1 / \tilde \delta) ) \max\{ B_1^2, B_2^2 \} \Big\} }{ \mu^2 R^2 }.$$
Then, with probability at least $1 - \tilde \delta$, it holds
$
\| \hat x_R(\bar y) - x_0 \| < R, 
\| \hat y_R(\bar x) - y_0 \| < R    .
$
\end{lem}

\begin{remark}
Lemma~\ref{lem:conditions_for_primaldual_gap} shows that if we properly set the values of $R$, $\eta_x$, $\eta_y$ and $T$, then $\hat x_R(\bar y)$ and $\hat y_R(\bar x)$ are the interior points of $\calB(x_0, R)$ and $\calB(y_0, R)$ with high probability.
Therefore, we conclude that $\hat x(\bar y) = \hat x_R(\bar y)$ and $\hat y(\bar x) = \hat y_R(\bar x)$ with probability $1 - \tilde \delta$ under the conditions of Lemma \ref{lem:conditions_for_primaldual_gap},
which allows us to derive the duality gap in LHS of (\ref{eq1:convergence_per_stage}) of Lemma \ref{lemma:convergence_RSPDsc_per_stage}.
\end{remark}

The following theorem gives the relation of duality gaps between two consecutive epochs of Algorithm~\ref{alg:RSPD_scsc} by using Lemma \ref{lemma:convergence_RSPDsc_per_stage} and the conditions proven by Lemma~\ref{lem:conditions_for_primaldual_gap}.

\begin{thm}\label{thm:connected_PD_gap_scsc}
Consider the $k$-th epoch of Algorithm~\ref{alg:RSPD_scsc} with an initial solution $(x_0^k, y_0^k)$ and the ending averaged solution $(\bar x_k, \bar y_k)$. 
Suppose Assumption~\ref{ass:2} holds and $\Gap (x_0^k, y_0^k) \leq \epsilon_{k-1}$.
Let 
$R_k \geq 2 \sqrt{ \frac{ 2 \epsilon_{k-1} }{ \min\{ \mu, \lambda \} } }$ (i.e. $\epsilon_{k-1} \leq \frac{ \min\{ \mu, \lambda \} R_k^2 }{ 8 }$),
$\eta_x^k = \frac{ \min\{ \mu, \lambda \} R_k^2 }{ 40 ( 5 + 3 \log(1/\tilde \delta) ) B_1^2 }$,
$\eta_y^k = \frac{ \min\{ \mu, \lambda \} R_k^2 }{ 40 ( 5 + 3 \log(1/\tilde \delta) ) B_2^2 }$ and
$$T_k \geq \frac{ \max \Big\{ 320^2 (B_1 + B_2)^2 3 \log( 1 / \tilde \delta ), 3200 ( 5 + 3 \log(1/\tilde \delta) ) \max\{ B_1^2, B_2^2 \} \Big\} }{ \min\{\mu, \lambda\}^2 R_k^2 }.$$
Then we have with probability $1 - \tildedelta$,
$
\Gap (\bar x_k, \bar y_k)
\leq
\frac{ \min\{ \mu, \lambda \} R_k^2}{16} $. 
\end{thm}

\begin{remark}
Theorem \ref{thm:connected_PD_gap_scsc} shows that after running $T_k$ iterations at the $k$-th stage, the upper bound of the duality gap would be halved with high probability, i.e., from $\frac{ \min\{ \mu, \lambda \} R_k^2}{8}$ to $\frac{ \min\{ \mu, \lambda \} R_k^2}{16} $.
Then, in order to make the duality gap of each outer loop of Algorithm \ref{alg:RSPD_scsc} halved from the last epoch, we can simply set $R_{k+1}^2 = \frac{R_k^2}{2}$,
and accordingly, $\eta_{x, k+1} = \frac{\eta_x^k}{2}$, $\eta_{y, k+1} = \frac{\eta_y^k}{2}$ and $T_{k+1} = 2T_k$.
\end{remark}

\begin{proof} 
(of Theorem \ref{thm:connected_PD_gap_scsc})
For any $x \in \calB(x_0^k, R_k)$ and $y \in \calB(y_0^k, R_k)$, we have $\| x - x_0^k \| \leq R$ and $\| y - y_0^k \| \leq R$,
so by (\ref{eq1:convergence_per_stage}) of Lemma \ref{lemma:convergence_RSPDsc_per_stage}, we have with probability $1 - \tilde \delta$
\begin{align}\label{eq:combined_all3_primaldual_gap}
f(\bar x_k, y) - f(x, \bar y_k) 
 \stackrel{(a)}{\leq} &
     \frac{R_k^2}{\eta_x^k T_k}
     + \frac{R_k^2}{\eta_y^k T_k}
     + \frac{ \eta_x^k B_1^2 }{ 2 } ( 5 + 3 \log(1 / \tilde \delta) )  
     + \frac{ \eta_y^k B_2^2 }{ 2 } ( 5 + 3 \log(1 / \tilde \delta) )  
     \nonumber\\
     &
     + \frac{ 4 ( B_1 + B_2 ) R_k \sqrt{3 \log ( 1 / \tilde \delta ) } }{ \sqrt{T_k} }
\stackrel{(b)}{\leq}
\frac{ \min \{ \mu, \lambda \} R_k^2}{16}  
,
\end{align}
where inequality $(a)$ is due to $x \in \calB(x_0^k, R_k)$ and $y \in \calB(y_0^k, R_k)$.
Inequality $(b)$ is due to the values of $\eta_x^k$, $\eta_y^k$ and $T_k$.
Recall the definitions $\hat x(\bar y_k) = \arg\min_{x \in X} f(x, \bar y_k)$ and $\hat y(\bar x_k) = \arg\max_{y \in Y} f(\bar x_k, y)$.
By Lemma \ref{lem:conditions_for_primaldual_gap}, we have $\hat x(\bar y_k) \in \calB(x_0^k, R_k)$ and $\hat y(\bar x_k) \in \calB(y_0^k, R_k)$ with probability $1 - \tilde \delta$.
Then from (\ref{eq:combined_all3_primaldual_gap}) we have
\begin{align*}
\Gap(\bar x_k, \bar y_k)
=
\max_{y \in Y} f(\bar x_k, y) - \min_{x \in X} f(x, \bar y_k) 
\leq
\frac{ \min \{ \mu, \lambda \} R_k^2}{16} .
\end{align*}
\end{proof}
Given the condition 
$
\Gap (x_0^k, y_0^k)
\leq
\epsilon_{k-1}
\leq
\frac{ \min\{ \mu, \lambda \} R_k^2}{8}  ,
$
we then conclude that running $T_k$ iterations in an epoch of Algorithm~\ref{alg:RSPD_scsc} would halve the duality gap with high probability.
As indicated in Theorem \ref{thm:connected_PD_gap_scsc}, the duality gap $\Gap (\bar x_k, \bar y_k)$ can be halved as long as the condition of Theorem \ref{thm:connected_PD_gap_scsc} holds. 
Then Theorem \ref{thm:tot_iter_scsc} is implied (the detailed proof is in Supplementary Materials).

\subsection{Proof Sketch of Theorem~\ref{thm:wcc_convergence} for the WCSC setting}

Due to limit of space, we only present a sketch here and present the full proof in the Supplement.  Recall $\hat f_k(x, y) = f(x, y) + \frac{ \gamma }{ 2 } \| x - x_0^k \|^2$. Let us denote its duality gap by $\widehat \Gap_k(x, y) = \hat f_k(x, \hat y_k(x)) - \hat f_k(\hat x_k(y), y)$, 
where we define $\hat y_k(x) := \arg\max_{y' \in Y} \hat f_k(x, y')$ given $x \in X$ and $\hat x_k(y) := \arg\min_{x' \in X} \hat f_k(x', y)$ given $y \in Y$.
Its saddle point solution is  denoted by $(\hat x^*_k, \hat y^*_k)$, i.e., $\hat f_k(\hat x^*_k, y) \leq \hat f_k(\hat x^*_k, \hat y^*_k) \leq \hat f_k(x, \hat y^*_k)$ for any $x \in X$ and $y \in Y$.
The key idea of our analysis is to connect the duality gap $\widehat \Gap_k(x_0^k , y_0^k)$ to $\gamma^2 \| \hat x_k^* - x_0^k \|^2$, 
and then by making $\gamma^2 \| \hat x_k^* - x_0^k \|^2 \leq \epsilon^2$, we can show that $x_0^k$ is a nearly $\epsilon$-stationary point.
To this end we first establish a bound of the duality gap for the regularized problem $\hat f_k(x, y)$ for the $k$-th epoch (Lemma \ref{lemma:inner_loop_wcc}).
 Then we connect  it to $\gamma \| \hat x_k^* - x_0^k \|^2$ (Lemma \ref{lemma:lhs_hat_gap}). Finally, we bound $\gamma \| \hat x_k^* - x_0^k \|^2$  by a telescoping sum of $\E [ \hatGap_k(x_0^k, y_0^k) ] - \E [ \hatGap_{k+1}(x_0^{k+1}, y_0^{k+1}) ] $ and $\E[P(x_0^k) - P(x_0^{k+1})]$.

\section{Conclusions}

In this paper, we filled the gaps between stochastic min-max and minimization optimization problems.
We proposed Epoch-GDA algorithms for general SCSC and general WCSC problems,
which do not impose any additional assumptions on the smoothness or the structure of the objective function.
Our key lemma provides sharp analysis of Epoch-GDA for both problems.
For SCSC min-max problems, to the best of our knowledge, our result is the first one to show that Epoch-GDA achieves the optimal rate of $O(1/T)$ for the duality gap of general SCSC min-max problems. 
For WCSC min-max problems, our analysis allows us to derive the best complexity $\tilde O(1/\epsilon^4)$ of Epoch-GDA to reach a nearly $\epsilon$-stationary point, which does not require smoothness, large mini-batch sizes or other structural conditions.


\bibliographystyle{plain}
\bibliography{all,spd}

\newpage
\appendix

\section{Proof of Theorem~\ref{thm:wcc_convergence} for the WCSC setting}
 Recall $\hat f_k(x, y) = f(x, y) + \frac{ \gamma }{ 2 } \| x - x_0^k \|^2$. Let us denote its duality gap by $\widehat \Gap_k(x, y) = \hat f_k(x, \hat y_k(x)) - \hat f_k(\hat x_k(y), y)$, 
where we define $\hat y_k(x) := \arg\max_{y' \in Y} \hat f_k(x, y')$ given $x \in X$ and $\hat x_k(y) := \arg\min_{x' \in X} \hat f_k(x', y)$ given $y \in Y$.
Its saddle point solution is  denoted by $(\hat x^*_k, \hat y^*_k)$, i.e., $\hat f_k(\hat x^*_k, y) \leq \hat f_k(\hat x^*_k, \hat y^*_k) \leq \hat f_k(x, \hat y^*_k)$ for any $x \in X$ and $y \in Y$.
The key idea of our analysis is to connect the duality gap $\widehat \Gap_k(x_0^k , y_0^k)$ to $\gamma^2 \| \hat x_k^* - x_0^k \|^2$, 
and then by making $\gamma^2 \| \hat x_k^* - x_0^k \|^2 \leq \epsilon^2$, we can show that $x_0^k$ is a nearly $\epsilon$-stationary point.
To this end we first establish a bound of the duality gap for the regularized problem $\hat f_k(x, y)$ for the $k$-th epoch (Lemma \ref{lemma:inner_loop_wcc}).
 Then we connect  it to $\gamma \| \hat x_k^* - x_0^k \|^2$ (Lemma \ref{lemma:lhs_hat_gap}). Finally, we bound $\gamma \| \hat x_k^* - x_0^k \|^2$  by a telescoping sum of $\E [ \hatGap_k(x_0^k, y_0^k) ] - \E [ \hatGap_{k+1}(x_0^{k+1}, y_0^{k+1}) ] $ and $\E[P(x_0^k) - P(x_0^{k+1})]$. 

\begin{lem}
\label{lemma:inner_loop_wcc}
Suppose Assumption~\ref{ass:3} holds and $\gamma = 2\rho$.
For $k \geq 1$, Lines \ref{alg:wcc:line:inner0} to \ref{alg:wcc:line:inner1} of Algorithm \ref{alg:wcc} guarantee
\begin{align}\label{eq:wcc_inner_loop}
&
\E [ \hatGap_k(\bar x_k, \bar y_k) ]
=
\E [ \max_{y \in Y} \hat f_k(\bar{x}_k, y) - \min_{x \in X} \hat f_k(x, \bar{y}_k) ]
=
\E [ \hat f_k(\bar x_k, \hat y_k(\bar x_k)) - \hat f_k(\hat x_k(\bar y_k), \bar y_k) ]
\nonumber\\
& \quad
\leq 
\frac{5 \eta_x^k M_1^2}{2} + \frac{5 \eta_y^k M_2^2}{2}
+ \frac{1}{T_k} \Big\{ ( \frac{1}{\eta_x^k} + \frac{\rho}{2} ) \E [ \| \hat x_k(\bar y_k) - x_0^k \|^2 ]
+ \frac{1}{\eta_y^k} \E [ \| \hat y_k(\bar x_k) - y_0^k \|^2 ] \Big\}  .
\end{align}
\end{lem}

For RHS of (\ref{eq:wcc_inner_loop}), particularly, due to $k \geq 1$, $T_k = \frac{106(k+1)}{3}$, $\eta_x^k = \frac{4}{\rho(k+1)}$ and $\eta_y^k = \frac{2}{\lambda (k+1)}$ in Algorithm \ref{alg:wcc}, 
we have $\frac{1}{T_k} ( \frac{1}{\eta_x^k} + \frac{\rho}{2} ) \leq \frac{3\rho}{212}$ and $\frac{1}{T_k \eta_y^k} = \frac{3\lambda}{212}$.
Then for the last two terms in the RHS of (\ref{eq:wcc_inner_loop}), we could have the following upper bound by the key lemma (Lemma \ref{lemma:key_lemma})
\begin{align}
\label{eq:rhs_hat_gap}
\frac{3}{53} \Big( \frac{\rho}{4} \| \hat x_k(\bar y_k) - x_0^k \|^2
+ \frac{\lambda}{4} \| \hat y_k(\bar x_k) - y_0^k \|^2 \Big)
\leq &
\frac{3}{53} \Big( 
\hatGap_k(x_0^k, y_0^k)
+ \hatGap_k(\bar x_k, \bar y_k) 
\Big).
\end{align}
On the other hand, the following lemma lower bounds LHS of (\ref{eq:wcc_inner_loop}) to construct telescoping sums.

\begin{lem}
\label{lemma:lhs_hat_gap}
We could derive the following lower bound for $\hatGap_k(\bar x_k, \bar y_k)$
\begin{align}
\label{eq:lhs_hat_gap}
\hatGap_k(\bar x_k, \bar y_k)
\geq
\frac{3}{50} \hatGap_{k+1}( x_0^{k+1}, y_0^{k+1} )
+ \frac{4}{5} ( P(x^{k+1}_0) - P(x_0^k) ) 
+ \frac{\gamma}{80} \| x_0^k - \hat x^*_k \|^2  .
\end{align}
\end{lem}

Lemma \ref{lemma:lhs_hat_gap} lower bounds $\widehat \Gap_k(\bar x_k, \bar y_k)$ in LHS of (\ref{eq:wcc_inner_loop}) by three parts.
The first part constructs telescoping sum of $\widehat \Gap_{k+1}(x_0^{k+1}, y_0^{k+1}) - \widehat \Gap_k(x_0^k, y_0^k)$ together with (\ref{eq:rhs_hat_gap}).
The second part itself is an element of telescoping sums over the primal gap.
The third part $\| x_0^k - \hat x_k^* \|^2$ can be used as the measure of nearly $\epsilon$-stationary point, which is further explored in Theorem \ref{thm:wcc_convergence}.

\begin{proof} 
(of Theorem \ref{thm:wcc_convergence})
Consider the $k$-th stage.
Let us start from (\ref{eq:wcc_inner_loop}) in Lemma \ref{lemma:inner_loop_wcc} as follows
\begin{align*}
&
\E [ \hatGap_k(\bar x_k, \bar y_k) ]
\nonumber\\
\leq &
\frac{5 \eta_x^k M_1^2}{2} + \frac{5 \eta_y^k M_2^2}{2}
+ \frac{1}{T_k} \Big\{ ( \frac{1}{\eta_x} + \frac{\rho}{2} ) \E [ \| \hat x_k(\bar y_k) - x_0^k \|^2 ]
+ \frac{1}{\eta_y} \E [ \| \hat y_k(\bar x_k) - y_0^k \|^2 ] \Big\}  
\nonumber\\
\stackrel{(a)}{\leq} &
\frac{5 \eta_x^k M_1^2}{2} + \frac{5 \eta_y^k M_2^2}{2}
+ \frac{3}{53} \big( \frac{\rho}{4} \E [ \| \hat x_k(\bar y_k) - x_0^k \|^2 ]
+ \frac{\lambda}{4} \E [ \| \hat y_k(\bar x_k) - y_0^k \|^2 ]  \Big) 
\nonumber\\
\stackrel{(\ref{eq:rhs_hat_gap})}{\leq} &
\frac{5 \eta_x^k M_1^2}{2} + \frac{5 \eta_y^k M_2^2}{2}
+ \frac{3}{53} \E [ \hatGap_k(x_0^k, y_0^k) ]
+ \frac{3}{53} \E [ \hatGap_k(\bar x_k, \bar y_k) ]  ,
\end{align*}
where
$(a)$ is due to settings 
$T_k = \frac{106(k+1)}{3}$,
$\eta_x^k = \frac{4 }{\rho (k+1)}$, and
$\eta_y^k = \frac{2 }{\lambda (k+1)}$.
Re-organizing the above inequality, we have
\begin{align}
\label{eq:wcc_inner_loop_after_upper_bound2}
\frac{50}{53} \E [ \hatGap_k(\bar x_k, \bar y_k) ]
\leq
\frac{5 \eta_x^k M_1^2}{2} + \frac{5 \eta_y^k M_2^2}{2}
+ \frac{3}{53} \E [ \hatGap_k(x_0^k, y_0^k) ]  .
\end{align}

Then for the LHS of (\ref{eq:wcc_inner_loop_after_upper_bound2}), we apply (\ref{eq:lhs_hat_gap}) of Lemma \ref{lemma:lhs_hat_gap} as follows
\begin{align}
\label{eq:wcc_plugin_lhs1}
&
\frac{50}{53} \Big( \frac{3}{50} \hatGap_{k+1}( x_0^{k+1}, y_0^{k+1} )
+ \frac{4}{5} ( P(x^{k+1}_0) - P(x_0^k) ) 
+ \frac{\gamma}{80} \| x_0^k - \hat x^*_k \|^2 \Big)
\nonumber\\
\leq &
\frac{5 \eta_x^k M_1^2}{2} + \frac{5 \eta_y^k M_2^2}{2}
+ \frac{3}{53} \E [ \hatGap_k(x_0^k, y_0^k) ]   .
\end{align}
Next we have
\begin{align}
\label{eq:wcc_plugin_lhs2}
\frac{5 \gamma}{424} \E [ \| x_0^k - \hat x_k^* \|^2 ]
\leq &
\frac{5 \eta_x^k M_1^2}{2} + \frac{5 \eta_y^k M_2^2}{2}
+ \frac{40}{53} \E [ P(x_0^k) - P(x_0^{k+1}) ]
\nonumber\\
&
+ \frac{3}{53} \Big( \E [ \hatGap_k(x_0^k, y_0^k) ] - \E [ \hatGap_{k+1}(x_0^{k+1}, y_0^{k+1}) ] \Big)
\end{align}

Summing from $k=1$ to $k=K$, we have
\begin{align}
\frac{5 \gamma}{424} \sum_{k=1}^K \E [ \| x_0^k - \hat x_k^* \|^2 ]
\leq &
\underbrace{ \sum_{k=1}^K \frac{5 \eta_x^k M_1^2}{2} + \sum_{k=1}^K \frac{5 \eta_y^k M_2^2}{2} }_{ := A }
+ \frac{40}{53} \underbrace{ \sum_{k=1}^K \E [ P(x_0^k) - P(x_0^{k+1}) ] }_{ := B }
\nonumber\\
\label{eq:wcc_final_telescope0}
&
+ \frac{3}{53} \underbrace{ \sum_{k=1}^K \Big( \E [ \hatGap_k(x_0^k, y_0^k) ] - \E [ \hatGap_{k+1}(x_0^{k+1}, y_0^{k+1}) ] \Big) }_{ :=C }
\\
\label{eq:wcc_final_telescope}
\leq &
5 \Big( \frac{ 2 M_1^2 }{\rho} + \frac{M_2^2}{\lambda} \Big) \ln(K + 1)
+ \frac{43}{53} \E [ \Gap(x_0^1, y_0^1) ]  ,
\end{align}
where the last inequality is due to the upper bounds the three terms $A$, $B$ and $C$ as follows.

\noindent
For the term $A$, we have
\begin{align*}
A 
= &
\sum_{k=1}^K \frac{5 \eta_x^k M_1^2}{2} + \sum_{k=1}^K \frac{5 \eta_y^k M_2^2}{2}
=
\frac{ 10 M_1^2 }{\rho} \sum_{k=1}^{K} \frac{1}{k + 1}
+ \frac{5 M_2^2}{\lambda} \sum_{k=1}^{K} \frac{1}{k + 1}
\\
\leq &
5 \Big( \frac{ 2 M_1^2 }{\rho} + \frac{M_2^2}{\lambda} \Big) \ln(K + 1)  ,
\end{align*}
where the second equality is due to the setting of $\eta_x^k = \frac{4}{\rho ( k + 1 ) }$ and $\eta_y^k = \frac{2}{\lambda ( k + 1 ) }$.
The last inequality is due to $\sum_{k=1}^{K+1} \frac{1}{k} \leq \ln(K + 1) + 1$.

For the term $B$, we have
\begin{align*}
B 
= &
\sum_{k=1}^K \E [ P(x_0^k) - P(x_0^{k+1}) ]
= 
\E [ P(x_0^1) - P(x_0^{K+1}) ]
=
\E [ f( x_0^1, \hat y(x_0^1)) - f(x_0^{K+1}, \hat y(x_0^{K+1}) ) ]
\\
\leq &
\E [ f( x_0^1, \hat y(x_0^1)) - f(x_0^{K+1}, y_0^1 ) ]
\leq
\E [ f( x_0^1, \hat y(x_0^1)) - f( \hat x(y_0^1), y_0^1 ) ]  
=
\E [ \Gap(x_0^1, y_0^1) ] ,
\end{align*}
where the two inequalities are due to $f(x_0^{K+1}, \hat y(x_0^{K+1}) ) \geq f(x_0^{K+1}, y_0^1 ) \geq f( \hat x(y_0^1), y_0^1 )$.

For the term $C$, we have
\begin{align*}
C 
= &
\sum_{k=1}^K \Big( \E [ \hatGap_k(x_0^k, y_0^k) - \hatGap_{k+1}(x_0^{k+1}, y_0^{k+1}) ] \Big)
\\
= &
\E [ \hatGap_k(x_0^1, y_0^1) - \hatGap_{K+1}(x_0^{K+1}, y_0^{K+1}) ]
\leq 
\E [ \hatGap_k(x_0^1, y_0^1) ]
\\
=
&
\E [ f(x_0^1, \hat y(x_0^1)) + \frac{\gamma}{2} \| x_0^1 - x_0^1 \|^2
- f(\hat x_1(y_0^1), y_0^1) - \frac{\gamma}{2} \| \hat x_1(y_0^1) - x_0^1 \|^2 ]
\\
\leq &
\E [ f(x_0^1, \hat y(x_0^1)) - f(\hat x(y_0^1), y_0^1) ]
=
\E [ \Gap(x_0^1, y_0^1) ] ,
\end{align*}
where the first inequality is due to $\hatGap_{K+1}(x_0^{K+1}, y_0^{K+1}) \geq 0$.
By plugging the above upper bounds of the three terms $A$, $B$ and $C$ into (\ref{eq:wcc_final_telescope0}), we have (\ref{eq:wcc_final_telescope}).

Then by randomly sampling $\tau$ from $\{ 1, ..., K \}$, we have
\begin{align*}
\E [ \| x_0^\tau - \hat x^*_\tau \|^2 ]
\leq &
\frac{424}{\gamma K} \Big( \frac{ 2 M_1^2 }{\rho} + \frac{M_2^2}{\lambda} \Big) \ln(K + 1)
+ \frac{344}{5 \gamma K} \E [ \Gap(x_0^1, y_0^1) ]  .
\end{align*}

\noindent
Since $\E [\text{Dist}(0, \partial P(\hat x^*_\tau))^2 ] \leq \gamma^2 \E [ \| x^*_\tau - x_0^\tau \|^2 ]$ and $\gamma = 2 \rho$, we could set
$$
K = \max \left\{ \frac{1696 \rho (\frac{2 M_1^2}{\rho} + \frac{M_2^2}{\lambda} ) }{ \epsilon^2 } \ln( \frac{1696 \rho (\frac{2 M_1^2}{\rho} + \frac{M_2^2}{\lambda} ) }{ \epsilon^2 }), 
\frac{ 1376 \rho \Gap(x_0^1, y_0^1) }{5 \epsilon^2} \right\}   ,
$$
which leads to $\gamma^2 \E[ \| x_\tau^* - x_0^\tau \|^2 ] \leq \epsilon^2$.
\noindent
Recall $T_k = \frac{106 (k+1)}{3}$.
To compute the total number of iterations, we have
\begin{align*}
T_{tot} 
= \sum_{k=1}^{K} T_k
= \frac{106}{3} \sum_{k=1}^{K} (k+1)
= O(K^2)
= \tilde O\left(\frac{1}{\epsilon^4}\right). 
\end{align*}
\end{proof}

\section{Proof of Lemma \ref{lemma:key_lemma}}

\begin{proof}

Let us first consider the first term in LHS of (\ref{eq:key_lemma}) as follows,
\begin{align}\label{eq:upper_bound1_key_lemma}
&
\frac{\mu}{4} \| \hat x(y_1) - x_0 \|^2
\nonumber\\
\leq & 
\frac{\mu}{2} \| \hat x(y_1) - x^* \|^2
+ \frac{\mu}{2} \| x^* - x_0^k \|^2
\nonumber\\
\stackrel{(a)}{\leq} &
f (x^*, y_1) - f (\hat x(y_1), y_1)
+ f (x_0, y^*) - f (x^*, y^*)
\nonumber\\
\stackrel{(b)}{\leq} &
f (x^*, y^*) - f (\hat x(y_1), y_1)
+ f (x_0, y^*) - f (x^*, y^*)
\nonumber\\
\stackrel{(c)}{\leq} &
f (x_0, \hat y(x_0)) - f (\hat x(y_1), y_1)  ,
\end{align}
where inequality $(a)$ is due to $\mu$-strong convexity of $f(x, y_1)$ in $x$ with fixed $y_1$ (with optimality at $\hat x(y_1)$) 
and 
$\mu$-strong convexity of $f(x, y^*)$ in $x$ with fixed $y^*$ (with optimality at $x^*$).
Inequality $(b)$ is due to $f (x^*, y_1) \leq f (x^*, y^*)$.
Inequality $(c)$ is due to $f (x_0, y^*) \leq f (x_0, \hat y(x_0))$.

In a similar way, for the second term, we have
\begin{align}\label{eq:upper_bound2_key_lemma}
&
\frac{\lambda}{4} \| \hat y(x_1) - y_0 \|^2
\nonumber\\
\leq &
\frac{\lambda}{2} \| \hat y(x_1) - y^* \|^2
+ \frac{\lambda}{2} \| y^* - y_0 \|^2
\nonumber\\
\stackrel{(a)}{\leq} &
f(x_1, \hat y(x_1)) - f(x_1, y^*)
+ f(x^*, y^*) - f(x^*, y_0)
\nonumber\\
\stackrel{(b)}{\leq} &
f(x_1, \hat y(x_1)) - f(x^*, y^*)
+ f(x^*, y^*) - f(x^*, y_0)
\nonumber\\
\stackrel{(c)}{\leq} &
f(x_1, \hat y(x_1)) - f(\hat x(y_0), y_0)   ,
\end{align}
where inequality $(a)$ is due to $\lambda$-strong concavity of $f(x_1, y)$ in $y$ with fixed $x_1$ (optimality at $\hat y(x_1)$) and 
$f(x^*, y)$ in $y$ with fixed $x^*$ (optimality at $\hat y^*$).
Inequality $(b)$ is due to $f(x_1, y^*) \geq f(x^*, y^*)$.
Inequality $(c)$ is due to $f(x^*, y_0) \geq f(\hat x(y_0), y_0)$.

Then, combining inequalities (\ref{eq:upper_bound1_key_lemma}) and (\ref{eq:upper_bound2_key_lemma}), we have
\begin{align*}
&
\frac{\mu}{4} \| \hat x(y_1) - x_0 \|^2
+ \frac{\lambda}{4} \| \hat y(x_1) - y_0 \|^2
\\
\leq &
f(x_0, \hat y(x_0)) - f(\hat x(y_1), y_1)
+ f(x_1, \hat y(x_1)) - f(\hat x(y_0), y_0)
\\
= &
\Big( \max_{y' \in \Omega_2} f(x_0, y') - \min_{x' \in \Omega_1} f(x', y_0) \Big)
+ \Big( \max_{y' \in \Omega_2} f(x_1, y') - \min_{x' \in \Omega_1} f(x', y_1) \Big)  .
\end{align*}
\end{proof}

\section{Proof of Lemma~\ref{lemma:convergence_RSPDsc_per_stage} }
\label{section_proof:lemma:convergence_RSPDsc_per_stage}

\begin{proof}
Before the proof, we first present the following two lemmas as follows.
\begin{lem}
\label{lemma:hp_bound_sum_grad}
Let $X_1, X_2, ..., X_T$ be independent random variables and $\E_t[ \exp( \frac{ X_t^2 }{ B^2 } ) ] \leq \exp(1)$ for any $t \in \{ 1, ..., T \}$.
Then we have with probability at least $1 - \tilde \delta$
\begin{align*}
\sum_{t=1}^T X_t \leq B^2 ( T + \log(1/\tilde \delta) )  .
\end{align*}
\end{lem}

\begin{lem}\label{lemma:inner_product_term_hp_bound} (Lemma 2 of \cite{lan2012validation})
Let $X_1, ..., X_T$ be a martingale difference sequence, i.e., $E_t [X_t] = 0$ for all $t$.
Suppose that for some values $\sigma_t$, for $t = 1, 2, ..., T$, we have $E_t [ \exp( \frac{  X^2_t }{ \sigma_t^2 } ) ] \leq \exp(1)$. 
Then with probability at least $1 - \delta$, we have
\begin{align*}
\sum_{t=1}^T X_t \leq \sqrt{ 3 \log(1/\delta) \sum_{t=1}^T \sigma_t^2 }  .
\end{align*}
\end{lem}

For simplicity of presentation, we use the notations $\Delta_{x}^{t} = \partial_{x} f(x_{t}, y_{t}; \xi_{t})$, $\Delta_{y}^{t} = \partial{y} f(x_{t}, y_{t}, ; \xi_{t})$, $\partial_{x}^{t} = \partial_{x} f(x_{t}, y_{t})$ and $\partial_{y}^{t} = \partial_{y} f(x_{t}, y_{t})$.
To prove Lemma~\ref{lemma:convergence_RSPDsc_per_stage}, we would leverage the following two update approaches:
\begin{align}
\label{eq1:two_update_sequences}
  & 
  \left \{ 
    \begin{array}{cc}
    x_{t+1} = \arg \min_{x \in X \cap \calB(x_0, R)} &  x^{\top} \Delta_{x}^{t} + \frac{1}{2 \eta_x } || x - x_{t} ||^{2}   \\
    y_{t+1} = \arg \min_{y \in Y \cap \calB(y_0, R)} & -y^{\top} \Delta_{y}^{t} + \frac{1}{2 \eta_y } || y - y_{t} ||^{2}   
    \end{array}
  \right.     \nonumber   \\
  & 
  \left \{ 
    \begin{array}{cc}
    \xtilde_{t+1} = \arg \min_{x \in X \cap \calB(x_0, R)} &  x^{\top} ( \partial_{x}^{t} - \Delta_{x}^{t} ) + \frac{1}{2 \eta_x } || x - \xtilde_{t} ||^{2}   \\
    \ytilde_{t+1} = \arg \min_{y \in Y \cap \calB(y_0, R)} & -y^{\top} ( \partial_{y}^{t} - \Delta_{y}^{t} ) + \frac{1}{2 \eta_y } || y - \ytilde_{t} ||^{2}    ,
    \end{array}
  \right. 
\end{align}
where $x_{0} = \xtilde_{0}$ and $y_{0} = \ytilde_{0}$.
The first two updates are identical to Line 4 and Line 5 in Algorithm~\ref{alg:RSPD_scsc}.
This can be verified easily.
Take the first one as example:
\begin{align*}
      x_{t+1} 
=  & 
      \Pi_{X } ( x_{t} - \eta_x \Delta_{x}^{t} )  
=     \arg\min_{x \in X \cap \calB(x_0, R)} || x - (x_{t} - \eta_x \Delta_{x}^{t}) ||^{2} 
\\
=  & \arg\min_{x \in X \cap \calB(x_0, R)} \frac{1}{2 \eta_x} || x - x_{t} ||^{2} + x^{\top} \Delta_{x}^{t}  .
\end{align*}
Let $\psi(x) = x^{\top} u + \frac{1}{2\gamma} || x - v ||^{2}$ with $x' = \arg \min_{x \in X'} \psi(x)$, which includes the four update approaches in~(\ref{eq1:two_update_sequences}) as special cases.
By using the strong convexity of $\psi(x)$ and the first order optimality condition ($\partial \psi(x')^{\top} ( x - x' ) \geq 0$), for any $x \in X'$, we have
\begin{align*}
       \psi(x) - \psi(x') \geq & \partial \psi(x')^T ( x - x') + \frac{1}{2\gamma} || x - x' ||^{2} \geq \frac{1}{2\gamma} || x - x' ||^{2},
\end{align*}
which implies
\begin{align*}
       0 \leq & ( x - x' )^{\top} u + \frac{1}{2\gamma}  || x - v ||^{2} - \frac{1}{2\gamma}  || x' - v ||^{2} - \frac{1}{2\gamma} || x - x' ||^{2}  \\
=    & ( v- x' )^{\top}  u- (v - x )^{\top} u + \frac{1}{2\gamma}  || x - v ||^{2} - \frac{1}{2\gamma} || x' - v ||^{2} - \frac{1}{2\gamma}  || x - x' ||^{2} \\
=    & - \frac{1}{2\gamma}  || x' - v ||^{2} + ( v - x' )^{\top} u  + \frac{1}{2\gamma} || x - v ||^{2}  - \frac{1}{2\gamma}  || x - x' ||^{2} - (v - x )^{\top} u  \\
\leq &  \frac{\gamma}{2} || u ||^{2} + \frac{1}{2\gamma}  || x - v ||^{2}  - \frac{1}{2\gamma} || x - x' ||^{2} - (v - x )^{\top} u. 
\end{align*}
Then
\begin{align}
\label{eq1:proof_expectation1}
(v - x )^{\top} u  \leq \frac{\gamma}{2} || u ||^{2} + \frac{1}{2\gamma}  || x - v ||^{2}  - \frac{1}{2\gamma}  || x - x' ||^{2} .
\end{align}
Applying the above result to the updates in~(\ref{eq1:two_update_sequences}), we have for any $x \in X \cap \calB(x_0, R)$ and $y \in Y \cap \calB(y_0, R)$,
\begin{align}
\label{eq1:seperate_all}
& ( x_{t} - x )^{\top} \Delta_{x}^{t} \leq \frac{1}{2 \eta_x} || x - x_{t} ||^{2} - \frac{1}{2 \eta_x} || x - x_{t+1} ||^{2} + \frac{\eta_x}{2} || \Delta_{x}^{t} ||^{2}   \nonumber\\
& ( y - y_{t} )^{\top} \Delta_{y}^{t} \leq \frac{1}{2 \eta_y} || y - y_{t} ||^{2} - \frac{1}{2 \eta_y} || y - y_{t+1} ||^{2} + \frac{\eta_y}{2} || \Delta_{y}^{t} ||^{2}   \nonumber\\
& ( \xtilde_{t} - x )^{\top} ( \partial_{x}^{t} - \Delta_{x}^{t} ) \leq \frac{1}{2 \eta_x} || x - \xtilde_{t} ||^{2} - \frac{1}{2 \eta_x} || x - \xtilde_{t+1} ||^{2} + \frac{\eta_x}{2} || \partial_{x}^{t} - \Delta_{x}^{t} ||^{2}   \nonumber\\
& ( y - \ytilde_{t} )^{\top} ( \partial_{y}^{t} - \Delta_{y}^{t} ) \leq \frac{1}{2 \eta_y} || y - \ytilde_{t} ||^{2} - \frac{1}{2 \eta_y} || y - \ytilde_{t+1} ||^{2} + \frac{\eta_y}{2} || \partial_{y}^{t} - \Delta_{y}^{t} ||^{2} .
\end{align}
Adding the above four inequalities together, we have
\begin{align}
\label{eq1:combine_all}
\LHS 
=    & 
       ( x_{t} - x )^{\top} \Delta_{x}^{t} + ( y - y_{t} )^{\top} \Delta_{y}^{t} + ( \xtilde_{t} - x )^{\top} ( \partial_{x}^{t} - \Delta_{x}^{t} ) + ( y - \ytilde_{t} )^{\top} ( \partial_{y}^{t} - \Delta_{y}^{t} )      
       \nonumber\\
=    & 
       ( x_{t} - x )^{\top} \partial_{x}^{t} + ( x_{t} - x )^{\top} ( \Delta_{x}^{t} - \partial_{x}^{t} )
       + ( y - y_{t} )^{\top} \partial_{y}^{t} + ( y - y_{t} )^{\top} ( \Delta_{y}^{t} - \partial_{y}^{t} )    
       \nonumber\\
     & + ( \xtilde_{t} - x )^{\top} ( \partial_{x}^{t} - \Delta_{x}^{t} ) + ( y - \ytilde_{t} )^{\top} ( \partial_{y}^{t} - \Delta_{y}^{t} )   
     \nonumber\\
=    & 
       - ( x - x_{t} )^{\top} \partial_{x}^{t} + (y - y_{t})^{\top} \partial_{y}^{t}  - ( x_{t} - \xtilde_{t} )^{\top} ( \partial_{x}^{t} - \Delta_{x}^{t} ) - ( \ytilde_{t} - y_{t} )^{\top} ( \partial_{y}^{t} - \Delta_{y}^{t} ) 
       \nonumber\\
\stackrel{(a)}{\geq} & 
       - ( f(x, y_{t}) - f(x_{t}, y_{t}) ) + ( f(x_{t}, y) - f(x_{t}, y_{t}) ) 
       - ( x_{t} - \xtilde_{t} )^{\top} ( \partial_{x}^{t} - \Delta_{x}^{t} ) - ( \ytilde_{t} - y_{t} )^{\top} ( \partial_{y}^{t} - \Delta_{y}^{t} )  
       \nonumber\\
=    &
       f(x_t, y) - f(x, y_t)
       - ( x_{t} - \xtilde_{t} )^{\top} ( \partial_{x}^{t} - \Delta_{x}^{t} ) - ( \ytilde_{t} - y_{t} )^{\top} ( \partial_{y}^{t} - \Delta_{y}^{t} )  
       \nonumber\\
\RHS
=    & \frac{1}{2\eta_x} \Big\{ || x - x_{t} ||^{2} - || x - x_{t+1} ||^{2} + || x - \xtilde_{t} ||^{2} - || x - \xtilde_{t+1} ||^{2} \Big\} 
       + \frac{\eta_x}{2} \Big\{ || \Delta_{x}^{t} ||^{2} + || \partial_{x}^{t} - \Delta_{x}^{t} ||^{2} \Big\}   \nonumber\\
     & + \frac{1}{2\eta_y} \Big\{  || y - y_{t} ||^{2} - || y - y_{t+1} ||^{2}   + || y - \ytilde_{t} ||^{2} - || y - \ytilde_{t+1} ||^{2} \Big\} 
       + \frac{\eta_y}{2} \Big\{ || \Delta_{y}^{t} ||^{2} + || \partial_{y}^{t} - \Delta_{y}^{t} ||^{2} \Big\}  
    \nonumber\\
\stackrel{ (b) }{ \leq }
     & \frac{1}{2\eta_x} \Big\{ || x - x_{t} ||^{2} - || x - x_{t+1} ||^{2} + || x - \xtilde_{t} ||^{2} - || x - \xtilde_{t+1} ||^{2} \Big\} 
       + \frac{\eta_x}{2} \Big\{ 3 || \Delta_{x}^{t} ||^{2} + 2 || \partial_{x}^{t} ||^{2} \Big\}   \nonumber\\
     & + \frac{1}{2\eta_y} \Big\{  || y - y_{t} ||^{2} - || y - y_{t+1} ||^{2}   + || y - \ytilde_{t} ||^{2} - || y - \ytilde_{t+1} ||^{2} \Big\} 
       + \frac{\eta_y}{2} \Big\{ 3 || \Delta_{y}^{t} ||^{2} + 2 || \partial_{y}^{t} ||^{2} \Big\}  
\end{align}
where inequality $(a)$ above is due to the convexity of $f(x, y_t)$ in $x$ and concavity of $f(x_t, y)$ in $y$.
Inequality $(b)$ is due to $(a+b)^2 \leq 2a^2 + 2b^2$.

Then we combine the LHS and RHS by summing up $t = 0, ..., T-1$:
\begin{align}
\label{eq1:combine_all2}
       \sum_{t=0}^{T-1} ( f(x_{t}, y) - f(x, y_{t}) )
\leq & \frac{1}{2\eta_x} \Big\{ || x - x_{0} ||^{2} - || x - x_{T} ||^{2} + || x - \xtilde_{0} ||^{2} - || x - \xtilde_{T} ||^{2}  \Big\} 
     \nonumber\\
     & \frac{1}{2\eta_y} \Big\{ || y - y_{0} ||^{2} - || y - y_{T} ||^{2} + || y - \ytilde_{0} ||^{2} - || y - \ytilde_{T} ||^{2} \Big\} 
     \nonumber\\
     & + \frac{3 \eta_x}{2} \underbrace{ \sum_{t=1}^T || \Delta_{x}^{t} ||^{2}  }_{ :=A }
       + \eta_x \underbrace{ \sum_{t=1}^T || \partial_{x}^{t} ||^{2} }_{ := B }
     \nonumber\\
     & + \frac{3\eta_y}{2} \underbrace{ \sum_{t=1}^T || \Delta_{y}^{t} ||^{2} }_{ := C} 
       + \eta_y \underbrace{ \sum_{t=1}^T || \partial_{y}^{t} ||^{2} }_{ := D}
     \nonumber\\
     & + \underbrace{ \sum_{t=0}^{T-1} \Big( ( x_{t} - \xtilde_{t} )^{\top} ( \partial_{x}^{t} - \Delta_{x}^{t} ) + ( y_{t} - \ytilde_{t} )^{\top} ( \partial_{y}^{t} - \Delta_{y}^{t} ) \Big) }_{ := E} .
\end{align}

In the following, we show how to bound the above $A$ to $E$ terms.
To bound the above term $A$ in (\ref{eq1:combine_all2}), we apply Lemma \ref{lemma:hp_bound_sum_grad} as follows, which holds with probability $1 - \tilde \delta$,
\begin{align}\label{eq:upper_bound_term_A}
\sum_{t=1}^T \| \Delta_x^t \|^2 \leq B_1^2 ( T + \log(1/\tilde \delta) )  .
\end{align}
Similarly, term $C$ in (\ref{eq1:combine_all2}) can be bounded with probability $1 - \tilde \delta$ as follows
\begin{align}\label{eq:upper_bound_term_C}
\sum_{t=1}^T \| \Delta_y^t \|^2 \leq B_2^2 ( T + \log(1/\tilde \delta) )  .
\end{align}

To bound term $B$ of (\ref{eq1:combine_all2}), which contains only the full subgradients $\partial_x^t$, we have
\begin{align*}
\| \partial_x^t \|^2
= 
\| \E[\Delta_x^t] \|^2
\leq 
\E[ \| \Delta_x^t \|^2 ]
\leq B^2_1   ,
\end{align*}
where the first inequality is due to Jensen's inequality and the second inequality is due to
\begin{align*}
\exp( \E[\frac{ \| \Delta_x^t \|^2 }{ B_1^2 }] )
\leq
\E [ \exp(\frac{ \| \Delta_x^t \|^2 }{ B_1^2 }) ] \leq \exp(1)
~~\Rightarrow~~
\E[ \frac{ \| \Delta_x^t \|^2 }{ B_1^2 } ] \leq 1
~~\Rightarrow~~
\E [ \| \Delta_x^t \|^2 ] \leq B_1^2   .
\end{align*}
Therefore, we have 
\begin{align}\label{eq:upper_bound_term_B}
\sum_{t=1}^T \| \partial_x^t \|^2 \leq T B_1^2  .
\end{align}
Similarly, for term $D$ in (\ref{eq1:combine_all2}), we have 
\begin{align}\label{eq:upper_bound_term_D}
\sum_{t=1}^T \| \partial_y^t \|^2 \leq T B_2^2  .
\end{align}

To bound term $E$ of (\ref{eq1:combine_all2}), let $U_t = ( x_t - \tilde x_t )^\top( \partial_x^t - \Delta_x^t )$ 
and $V_t = ( y_t - \tilde y_t )^\top( \partial_y^t - \Delta_y^t )$ for $t \in \{ 0, ..., T-1 \}$, which are Martingale difference sequences.
We thus would like to use Lemma \ref{lemma:inner_product_term_hp_bound} to handle these terms.
To this end, we can first upper bound $|U_t|$ and $|V_t|$ as follows
\begin{align*}
|U_t| 
= &
| ( x_t - \tilde x_t )^\top( \partial_x^t - \Delta_x^t ) |
\leq 
\| x_t - x_0 + x_0 - \tilde x_t \| \cdot \| \partial_x^t - \Delta_x^t \|
\nonumber\\
\leq &
2R ( \| \partial_x^t \| + \| \Delta_x^t \| )
\leq
2R ( B_1 + \| \Delta_x^t \| )   ,
\nonumber\\
| V_t |
= &
| ( y_t - \tilde y_t )^\top( \partial_y^t - \Delta_y^t ) |
\leq 
\| ( y_t - y_0 + y_0 - \tilde y_t \| \cdot \| \partial_y^t - \Delta_y^t ) \|
\nonumber\\
\leq &
2R ( \| \partial_y^t \| + \| \Delta_y^t ) \| )
\leq 
2R ( B_2 + \| \Delta_y^t ) \| )    .
\end{align*}
Then the above two inequalities implies that
\begin{align}\label{eq:inner_product_term_hp_bound_1}
\E_t[ \exp( \frac{ U_t^2 }{ 16 B_1^2 R^2 } ) ]
\leq &
\E_t[ \exp( \frac{ ( 2R ( B_1 + \| \Delta_x^t \| ) )^2 }{ 16 B_1^2 R^2 } ) ]
\stackrel{ ( a ) }{ \leq } 
\E_t[ \exp( \frac{ 4R^2 ( 2 B_1^2 + 2 \| \Delta_x^t \|^2 )  }{ 16 B_1^2 R^2 } ) ]
\nonumber\\
= &
\E_t[ \exp( \frac{ B_1^2 + \| \Delta_x^t \|^2   }{ 2 B_1^2 } ) ]
=
\E_t[ \exp( \frac{1}{2} + \frac{ \| \Delta_x^t \|^2   }{ 2 B_1^2 } ) ]
\nonumber\\
= &
\exp( \frac{1}{2} ) \cdot \E_t \Big[ \sqrt{ \exp( \frac{ \| \Delta_x^t \|^2 }{ B_1^2 } ) } \Big]
\stackrel{ ( b ) }{ \leq }
\exp( \frac{1}{2} ) \cdot \sqrt{ \E_t[ \exp( \frac{ \| \Delta_x^t \|^2   }{ B_1^2 } ) ] }
\nonumber\\
\stackrel{ ( c ) }{ \leq } &
\exp( \frac{1}{2} ) \sqrt{ \exp(1) }
= 
\exp(1)   ,
\end{align}
where inequality $(a)$ is due to $(a+b)^2 \leq 2 a^2 + 2 b^2$,
inequality $(b)$ is due to the concavity of $\sqrt{\cdot}$ and Jensen's inequality.
Inequality $(c)$ is due to the assumption.
In a similar way, we have
\begin{align}\label{eq:inner_product_term_hp_bound_2}
\E_t [ \exp ( \frac{ V_t^2 }{ 16 B_2^2 R^2 } ) ] 
\leq 
\exp(1).
\end{align}

Next, applying Lemma \ref{lemma:inner_product_term_hp_bound} with (\ref{eq:inner_product_term_hp_bound_1}) and (\ref{eq:inner_product_term_hp_bound_2}), we have with probability at least $1 - \tilde \delta$
\begin{align}\label{eq:inner_product_upper_bound}
\sum_{t=0}^{T-1} U_t \leq 4 B_1 R \sqrt{3 \log( 1 / \tilde \delta ) T}  ,
\nonumber\\
\sum_{t=0}^{T-1} V_t \leq 4 B_2 R \sqrt{3 \log( 1 / \tilde \delta ) T}  .
\end{align}

For LHS of (\ref{eq1:combine_all2}), by Jensen's inequality, we have
\begin{align}\label{eq:jensens}
\sum_{t=0}^{T-1} ( f(x_{t}, y) - f(x, y_{t}) )
\geq 
T ( f(\bar x, y) - f(x, \bar y) ) ,
\end{align}
where $\bar x = \frac{1}{T} \sum_{t=0}^{T-1} x_t$ and $\bar y = \frac{1}{T} \sum_{t=0}^{T-1}$.


Suppose $T \geq 1$.
By plugging (\ref{eq:jensens}), (\ref{eq:upper_bound_term_A}), (\ref{eq:upper_bound_term_C}), (\ref{eq:upper_bound_term_B}), (\ref{eq:upper_bound_term_D}) and (\ref{eq:inner_product_upper_bound}) back into (\ref{eq1:combine_all2}), with probability at least $1 - \tilde \delta$, we have
\begin{align}
\label{eq1:combine_all3}
       f(\bar x, y) - f(x, \bar y) 
\leq & \frac{ \| x - x_{0} \|^{2} }{ \eta_x T } 
       + \frac{ \| y - y_{0} \|^{2} }{ \eta_y T } 
       + \frac{ \eta_x B_1^2 }{ 2 } ( 5 + 3 \log(1/\tilde \delta) )
       + \frac{ \eta_y B_2^2 }{ 2 } ( 5 + 3 \log(1/\tilde \delta) )
     \nonumber\\
     & + \frac{ 4 ( B_1 + B_2 ) R \sqrt{ 3 \log(1/\tilde\delta) } }{ \sqrt{T} } 
\end{align}

\end{proof}

\section{Proof of Lemma \ref{lemma:hp_bound_sum_grad} }

\begin{proof}  
First, we start from
\begin{align*}
\E[ \exp( \frac{ \sum_{t=1}^T X_t }{ B^2 } ) ]
= &
\E[ \E_T[ \exp( \frac{ \sum_{t=1}^{T-1} X_t + X_T }{ B^2 } ) ] ]
\nonumber\\
= &
\E[ \exp( \frac{ \sum_{t=1}^{T-1} X_t }{ B^2 } ) \cdot \E_T[ \exp(\frac{X_T}{B^2}) ] ]
\nonumber\\
\leq &
\E[ \exp( \frac{ \sum_{t=1}^{T-1} X_t }{ B^2 } ) \cdot \exp(1) ]
\nonumber\\
\leq &
\E[ \exp( \frac{ \sum_{t=1}^{T-2} X_t }{ B^2 } ) \cdot \exp(2) ]
\nonumber\\
\leq & 
\exp(T)    ,
\end{align*}
where the first inequality is due to the assumption.

Markov inequality indicates that $P( X \geq a ) \leq \frac{ \E[X] }{ a }$ for a random variable $X$,
which, by additionally introducing $\tilde \delta$, leads to 
\begin{align*}
P \Big( \exp( \frac{ \sum_{t=1}^T X_t }{ B^2 } ) \geq \frac{ \E[ \exp( \frac{ \sum_{t=1}^T X_t }{ B^2 } ) ] }{ \tilde \delta } \Big) 
\leq 
\tilde \delta  .
\end{align*}
Therefore, with probability at least $1 - \tilde \delta$, we have
\begin{align*}
&
\exp( \frac{ \sum_{t=1}^T X_t }{ B^2 } ) 
\leq 
\frac{ \E [ \exp( \frac{ \sum_{t=1}^T X_t }{ B^2 } ) ] }{ \tilde \delta }
\leq
\frac{ \exp(T) }{ \tilde \delta }
\\
\Rightarrow &
\frac{ \sum_{t=1}^T X_t }{ B^2 }
\leq 
\log( \frac{\exp(T) }{ \tilde \delta } )
=
\log( \exp(T) ) + \log( 1 / \tilde \delta )
=
T + \log( 1 / \tilde \delta )   
\\
\Rightarrow &
\sum_{t=1}^T X_t 
\leq 
B^2 ( T + \log( 1 / \tilde \delta ) )  .
\end{align*}
\end{proof}

\section{Proof of Lemma~\ref{lem:conditions_for_primaldual_gap} }
\label{section_proof:lem:conditions_for_primaldual_gap}

\begin{proof}

Here we consider the following problem
\begin{align*}
\min_{x \in X \cap \calB(x_0, R)} \max_{y \in Y \cap \calB(y_0, R)} f(x, y) 
\end{align*}
with two solutions $(x_0, y_0)$ and $(\bar x, \bar y)$.

By (\ref{lemma:key_lemma}) of Lemma \ref{lemma:key_lemma}, 
we have 
\begin{align}\label{eq:key_lemma_to_scsc}
\frac{\mu}{4} \| \hat x_R(\bar y) - x_0 \|^2
+ \frac{\lambda}{4} \| \hat y_R(\bar x) - y_0 \|^2
\leq &
\underbrace{ \max_{y' \in Y \cap \calB(y_0, R)} f(x_0, y') - \min_{x \in X \cap \calB(x_0, R)} f(x', y_0) }_{ := A }
\nonumber\\
&
+ \underbrace{ \max_{y' \in Y \cap \calB(y_0, R)} f(\bar x, y') - \min_{x \in X \cap \calB(x_0, R)} f(x', \bar y) }_{ := B }   .
\end{align}

We can bound the above term $A$ as follows
\begin{align}\label{eq:key_lemma_to_scsc_term_A}
&
\max_{ y' \in Y \cap \calB(y_0, R) } f(x_0, y') - \min_{ x \in X \cap \calB(x_0, R) } f(x', y_0)
\nonumber\\
\leq &
\max_{ y' \in Y } f(x_0, y') - \min_{ x \in X } f(x', y_0)
\leq 
\frac{ \min\{ \mu, \lambda \} R^2 }{ 8 }   ,
\end{align}
where the last inequality is due to the setting of $R$.

Recall the definitions 
$$
\hat x_R(\bar y) = \arg\min_{x' \in x \cap \calB(x_0, R)} f(x', \bar y)  ,
$$ 
$$
\hat y_R(\bar x) = \arg\max_{y' \in Y \cap \calB(y_0, R)} f(\bar x, y')  .
$$
To Bound term $B$ in (\ref{eq:key_lemma_to_scsc}), we apply Lemma \ref{lemma:convergence_RSPDsc_per_stage} as follows
\begin{align}\label{eq:key_lemma_to_scsc_term_B}
&
\max_{y' \in Y \cap \calB(y_0, R)} f(\bar x, y') - \min_{x' \in \cap \calB(x_0, R)} f(x, \bar y)
\nonumber\\
\leq & 
\frac{ \| \hat x_R(\bar y) - x_0 \|^2 }{ \eta_x T }
+ \frac{ \| \hat y_R(\bar x) - y_0 \|^2 }{ \eta_y T }
+ \frac{ \eta_x B_1^2 }{ 2 } ( 5 + 3 \log(1 / \tilde \delta) )
+ \frac{ \eta_y B_2^2 }{ 2 } ( 5 + 3 \log(1 / \tilde \delta) )
\nonumber\\
&
+ \frac{ 4 ( B_1 + B_2 ) R \sqrt{2 \log(1 / \tilde \delta)} }{ \sqrt{T} }
\nonumber\\
\leq &
\frac{ R^2 }{ \eta_x T }
+ \frac{ R^2 }{ \eta_y T }
+ \frac{ \eta_x B_1^2 }{ 2 } ( 5 + 3 \log(1 / \tilde \delta) )
+ \frac{ \eta_y B_2^2 }{ 2 } ( 5 + 3 \log(1 / \tilde \delta) )
\nonumber\\
&
+ \frac{ 4 ( B_1 + B_2 ) R \sqrt{2 \log(1 / \tilde \delta)} }{ \sqrt{T} }
\nonumber\\
\leq &
\frac{ \min\{ \mu, \lambda \} R^2 }{ 16 }   ,
\end{align}
where the last inequality holds with probability at least $1 - \tilde \delta$ with the setting of $\eta_x$, $\eta_y$ and $T$ as follows
\begin{align}\label{eq:set_eta_and_T}
\eta_x = &
           \frac{ \min \{ \mu, \lambda \} R^2}{ 40 ( 5 + 3 \log(1/\tilde \delta) ) B_1^2 }    ,
\eta_y =
           \frac{ \min \{ \mu, \lambda \} R^2}{ 40 ( 5 + 3 \log(1/\tilde \delta) ) B_2^2 }
\nonumber\\
T \geq &
           \frac{ \max \Big\{ 
           320^2 (B_1 + B_1)^2 3 \log( 1 / \tilde \delta ), 
           3200 ( 5 + 3 \log(1/\tilde \delta) ) \max\{ B_1^2, B_2^2 \} 
           \Big\} }{ \min \{ \mu, \lambda \}^2 R^2 }   .
\end{align}

Finally, we use (\ref{eq:key_lemma_to_scsc_term_A}) and (\ref{eq:key_lemma_to_scsc_term_B}) to bound term $A$ and term $B$ in (\ref{eq:key_lemma_to_scsc}) as follows
\begin{align*}
\frac{ \mu }{ 4 } \| \hat x_R(\bar y) - x_0 \|^2 
+ \frac{ \lambda }{ 4 } \| \hat y_R(\bar x) - y_0 \|^2
\leq &
\frac{ \min\{ \mu, \lambda \} R^2 }{ 8 }
+ \frac{ \min\{ \mu, \lambda \} R^2 }{ 16 }
=
\frac{ 3 \min\{ \mu, \lambda \} R^2 }{ 16 }
\\
< &
\frac{ \min\{ \mu, \lambda \} R^2 }{ 4 }   .
\end{align*}
It implies 
$$
\| \hat x_R(\bar y) - x_0 \| < R  ,
$$
$$
\| \hat y_R(\bar x) - y_0 \| < R  ,
$$
which shows $\hat x_R(\bar y)$ and $\hat y_R(\bar y)$ are interior points of $\calB(x_0, R)$ and $\calB(y_0, R)$, respectively,
so that $\hat x_R(\bar y) = \hat x(\bar y)$ and $\hat y_R(\bar x) = \hat y(\bar x)$.

\end{proof}

\section{Proof of Theorem~\ref{thm:tot_iter_scsc} }
\label{section_proof:thm:tot_iter_scsc}

\begin{proof}
Let 
\begin{align*}
T_1 = &
        \frac{ \max \Big\{ 320^2 (B_1 + B_2)^2 3 \log( 1 / \tildedelta ), 3200 ( 3 \log( 1 / \tilde \delta ) + 2 ) \max\{ B_1^2, B_2^2 \} \Big\} }{ \min \{ \mu, \lambda \}^2 R_1^2 }   ,
\end{align*}
where $\Gap(x_0, y_0) = \max_{y \in Y} f(x_0, y) - \min_{x \in X} f(x, y_0) 
\leq
\epsilon_0$
and 
$R_1 \geq 2 \sqrt{ \frac{ 2 \epsilon_0 }{ \min\{ \mu, \lambda \} } }.$

Given $T_{k+1} = 2T_k$ in Algorithm \ref{alg:RSPD_scsc} and $K = \lceil \log( \frac{ \epsilon_0 }{ \epsilon } ) \rceil$, the total number of iterations can be computed by
\begin{align*}
T_{tot} 
= &
\sum_{k=1}^{K} T_k 
=      
T_1 \sum_{k=1}^{K} 2^{k - 1} 
=
T_1 ( 2^K - 1 )
\leq
T_1 2^{ \lceil \log(\frac{\epsilon_0}{\epsilon}) \rceil } 
\leq 
T_1 \frac{ 2 \epsilon_0 }{ \epsilon }
\nonumber\\
= &
\frac{ \max \Big\{ 320^2 (B_1 + B_2)^2 3 \log( 1 / \tildedelta ), 3200 ( 3 \log( 1 / \tilde \delta ) + 2 ) \max\{ B_1^2, B_2^2 \} \Big\} }{ \min \{ \mu, \lambda \}^2 R_1^2 } 
\cdot \frac{ 2 \epsilon_0 }{ \epsilon } 
\nonumber\\
\leq & 
\frac{ \max \Big\{ 320^2 (B_1 + B_2)^2 3 \log( 1 / \tildedelta ), 3200 ( 3 \log( 1 / \tilde \delta ) + 2 ) \max\{ B_1^2, B_2^2 \} \Big\} }{ 8 \min \{ \mu, \lambda \} \epsilon_0 } 
\cdot \frac{ 2 \epsilon_0 }{ \epsilon } 
\nonumber\\
= &
\frac{ \max \Big\{ 320^2 (B_1 + B_2)^2 3 \log(\frac{1}{\tildedelta}), 3200 ( 3 \log(1/\tilde \delta) + 2 ) \max\{ B_1^2, B_2^2 \} \Big\} }{ 4 \min\{ \mu, \lambda \} \epsilon }
\end{align*}
\end{proof}

\section{Proof of Lemma~\ref{lemma:inner_loop_wcc} }
\label{section_proof:lemma:inner_loop_wcc}

\begin{proof}
In this proof, we focus on the analysis of one inner loop and thus omit the index of $k$ for simpler presentation.
Let $\Delta_x^t = \partial_x f(x_t, y_t; \xi_t)$, $\Delta_y^t = \partial_y f(x_t, y_t; \xi^t)$, 
$\partial_x^t = \partial_x f(x_t, y_t)$ and $\partial_y^t = \partial_y f(x_t, y_t)$.
Denote $\hat f(x, y) = f(x, y) + \frac{\gamma}{2} \| x - x_0 \|^2$.

Let $\psi_x^t(x) = x^\top \Delta_x^t + \frac{1}{2 \eta_x} \| x - x_t \|^2 + \frac{\gamma}{2} \| x - x_0 \|^2$
and
$\psi_y^t(y) = -y^\top \Delta_y^t + \frac{1}{2 \eta_y} \| y - y_t \|^2$.
According to the update of $x_{t+1}$ and $y_{t+1}$, we have
$x_{t+1} = \arg\min_{x \in X} \psi_x^t(x)$ and $y_{t+1} = \arg\max_{y \in Y} \psi_y^t(y)$.
It is easy to verify that $\psi_x^t$ and $\psi_y^t$ are strongly convex in $x$ and $y$, respectively.

By $\Big(\frac{1}{\eta_x} + \gamma\Big)$-strong convexity of $\psi_x^t(x)$ and the optimality condition at $x_{t+1}$, we have
\begin{align}\label{eq:update_strong_convex_x}
&
\Big( \frac{1}{2 \eta_x} + \frac{\gamma}{2} \Big) \| x - x_{t+1} \|^2
\leq 
\psi_x^t(x) - \psi_x^t(x_{t+1})
\nonumber\\
=
&
x^\top \Delta_x^t + \frac{1}{2\eta_x} \| x - x_t \|^2 + \frac{\gamma}{2} \| x - x_0 \|^2
- \Big(  x_{t+1}^\top \Delta_x^t + \frac{1}{2\eta_x} \| x_{t+1} - x_t \|^2 + \frac{\gamma}{2} \| x_{t+1} - x_0 \|^2  \Big)
\nonumber\\
= &
( x - x_t )^\top \partial_x^t 
+ ( x_t - x_{t+1} )^\top \partial_x^t
+ ( x - x_{t+1} )^\top ( \Delta_x^t - \partial_x^t )
\nonumber\\
&
+ \frac{1}{2\eta_x} \| x - x_t \|^2
+ \frac{\gamma}{2} \| x - x_0 \|^2
- \frac{1}{2\eta_x} \| x_{t+1} - x_t \|^2
- \frac{\gamma}{2} \| x_{t+1} - x_0 \|^2
\nonumber\\
\stackrel{(a)}{\leq}
&
f(x, y_t) - f(x_t, y_t) + \frac{\gamma}{2} \| x - x_0 \|^2 - \frac{\gamma}{2} \| x_t - x_0 \|^2
+ \frac{\gamma}{2} \Big( \| x_t - x_0 \|^2 - \| x_{t+1} - x_0 \|^2  \Big)
\nonumber\\
&
+ \Big( \frac{1}{2\eta_x} + \frac{\rho}{2} \Big) \| x - x_t \|^2
+ ( x - x_t )^\top ( \Delta_x^t - \partial^t )
+ ( x_t - x_{t+1} )^\top \Delta_x^t
- \frac{1}{2\eta_x} \| x_{t+1} - x_t \|^2
\nonumber\\
\stackrel{(b)}{\leq}
&
\hat f(x, y_t) - \hat f(x_t, y_t) 
+ \frac{\gamma}{2} \Big( \| x_t - x_0 \|^2 - \| x_{t+1} - x_0 \|^2  \Big)
\nonumber\\
&
+ \Big( \frac{1}{2\eta_x} + \frac{\rho}{2} \Big) \| x - x_t \|^2
+ ( x - x_t )^\top ( \Delta_x^t - \partial^t )
+ \frac{\eta_x}{2} \| \Delta_x^t \|^2  ,
\end{align}
where inequality $(a)$ is due to $\rho$-weakly convexity of $f$ in $x$.
Inequality $(b)$ is due to Young's inequality,
i.e.,
$( x_t - x_{t+1} )^\top \Delta_x^t
- \frac{1}{2\eta_x} \| x_{t+1} - x_t \|^2 \leq \frac{\eta_x}{2} \| \Delta_x^t \|^2$.

Similarly, due to the $\frac{1}{\eta_y}$-strong convexity of $\psi_y^t$ in $y$ and the optimality condition of $y_{t+1}$, we have
\begin{align}\label{eq:update_strong_convex_y}
&
\frac{1}{2\eta_y} \| y - y_{t+1} \|^2
\leq
\psi_y^t(y) - \psi_y^t(y_{t+1})
\nonumber\\
= &
-y^\top \Delta_y^t + \frac{1}{2 \eta_y} \| y - y_t \|^2
- \Big(  -y_{t+1}^\top \Delta_y^t + \frac{1}{2 \eta_y} \| y_{t+1} - y_t \|^2  \Big)
\nonumber\\
= &
( y_t - y )^\top \partial_y^t 
+ ( y_{t+1} - y_t )^\top \partial_y^t
+ ( y_{t+1} - y )^\top(\Delta_y^t - \partial_y^t)
\nonumber\\
&
+ \frac{1}{2\eta_y} \| y - y_t \|^2
- \frac{1}{2\eta_y} \| y_{t+1} - y_t \|^2
\nonumber\\
\stackrel{(a)}{\leq} &
f(x_t, y_t) - f(x_t, y)
+ ( y_{t+1} - y_t )^\top \Delta_y^t
+ ( y_t - y )^\top ( \Delta_y^t - \partial_y^t )
\nonumber\\
&
+ \frac{1}{2\eta_y} \| y - y_t \|^2
- \frac{1}{2\eta_y} \| y_{t+1} - y_t \|^2
\nonumber\\
\stackrel{(b)}{\leq}
&
\hat f(x_t, y_t) - \hat f(x_t, y)
+ ( y_t - y )^\top ( \Delta_y^t - \partial_y^t ) 
+ \frac{1}{2\eta_y} \| y - y_t \|^2
+ \frac{\eta_y}{2} \| \Delta_y^t \|^2   ,
\end{align}
where inequality $(a)$ is due to concavity of $f$ in $y$.
Inequality $(b)$ is due to Young's inequality, i.e.,
$( y_{t+1} - y_t )^\top \Delta_y^t - \frac{1}{2\eta_y} \| y_{t+1} - y_t \|^2 \leq \frac{\eta_y}{2} \| \Delta_y^t \|^2.$

Combining (\ref{eq:update_strong_convex_x}) and (\ref{eq:update_strong_convex_y}), we have
\begin{align}\label{eq:inner_combine1}
&
\hat f(x_t, y) - \hat f(x, y_t)
\leq 
\frac{\eta_x}{2} \| \Delta_x^t \|^2
+ 
\frac{\eta_y}{2} \| \Delta_y^t \|^2
\nonumber\\
&
+ ( x - x_t )^\top( \Delta_x^t - \partial^t )
+ ( y_t - y )^\top ( \Delta_y^t - \partial^t )
+ \frac{\gamma}{2} \Big( \| x_t - x_0 \|^2 - \| x_{t+1} - x_0 \|^2 \Big)
\nonumber\\
&
+ \Big( \frac{1}{2\eta_x} + \frac{\rho}{2} \Big) \| x - x_t \|^2 - \Big( \frac{1}{2\eta_x} + \frac{\gamma}{2} \Big) \| x - x_{t+1} \|^2 
+ \frac{1}{2\eta_y} \Big( \| y - y_t \|^2 - \| y - y_{t+1} \|^2 \Big)   .
\end{align}

Now we do not take expectation, since we aim to eliminate the randomness of $x$ and $y$ in 
$( x - x_t )$ and
$( y_t - y )$, respectively.
To achieve this, we use the following updates
\begin{align*}
\tilde x_{t+1} & = \arg\min_{x \in X}  x^\top ( \partial_x^t - \Delta_x^t ) + \frac{1}{2\eta_x} \| x - \tilde x_t \|^2
\\
\tilde y_{t+1} & = \arg\min_{y \in Y} -y^\top ( \partial_y^t - \Delta_y^t ) + \frac{1}{2\eta_y} \| y - \tilde y_t \|^2   ,
\end{align*}
where $\tilde x_0 = x_0$ and $\tilde y_0 = y_0$.

Using similar analysis as the beginning, we have
\begin{align*}
\frac{1}{2\eta_x} \| x - \tilde x_{t+1} \|^2
\leq
&
x^\top ( \partial_x^t - \Delta_x^t ) + \frac{1}{2\eta_x} \| x - \tilde x_t \|^2
- \Big( \tilde x_{t+1}^\top ( \partial_x^t - \Delta_x^t ) + \frac{1}{2\eta_x} \| \tilde x_{t+1} - \tilde x_t \|^2 \Big)
\\
=
&
( \tilde x_t - x )^\top( \Delta_x^t - \partial_x^t ) + \frac{1}{2\eta_x} \| x - \tilde x_t \|^2
+ ( \tilde x_t - \tilde x_{t+1} )^\top( \partial_x^t - \Delta_x^t ) 
- \frac{1}{2\eta_x} \| \tilde x_{t+1} - \tilde x_t \|^2
\\
\leq &
( \tilde x_t - x )^\top( \Delta_x^t - \partial_x^t ) + \frac{1}{2\eta_x} \| x - \tilde x_t \|^2
+ \frac{\eta_x}{2} \| \partial_x^t - \Delta_x^t \|^2
\\
\leq &
( \tilde x_t - x )^\top( \Delta_x^t - \partial_x^t ) + \frac{1}{2\eta_x} \| x - \tilde x_t \|^2
+ \eta_x \| \partial_x^t \|^2 + \eta_x \| \Delta_x^t \|^2  .
\end{align*}

We could also derive the similar result for $y$ as follows
\begin{align*}
\frac{1}{2\eta_y} \| y - \tilde y_{t+1} \|^2
\leq
&
- y^\top ( \partial_y^t - \Delta_y^t ) + \frac{1}{2\eta_y} \| y - \tilde y_t \|^2
- \Big( - \tilde y_{t+1}^\top ( \partial_y^t - \Delta_y^t ) 
        + \frac{1}{2\eta_y} \| \tilde y_{t+1} - \tilde y_t \|^2 \Big)
\\
=
&
( y - \tilde y_t )^\top( \Delta_y^t - \partial_y^t ) 
+ \frac{1}{2\eta_y} \| y - \tilde y_t \|^2
+ ( \tilde y_{t+1} - \tilde y_t )^\top( \partial_y^t - \Delta_y^t ) 
- \frac{1}{2\eta_y} \| \tilde y_{t+1} - \tilde y_t \|^2
\\
\leq &
( y - \tilde y_t )^\top( \Delta_y^t - \partial_y^t ) + \frac{1}{2\eta_y} \| y - \tilde y_t \|^2
+ \frac{\eta_y}{2} \| \partial_y^t - \Delta_y^t \|^2
\\
\leq &
( y - \tilde y_t )^\top( \Delta_y^t - \partial_y^t ) + \frac{1}{2\eta_y} \| y - \tilde y_t \|^2
+ \eta_y \| \partial_y^t \|^2 + \eta_y \| \Delta_y^t \|^2  .
\end{align*}

Summing the above two inequalities, we have
\begin{align}\label{eq:inner_combine2}
0
\leq
&
\frac{1}{2\eta_x} \Big( \| x - \tilde x_t \|^2 - \| x - \tilde x_{t+1} \|^2\Big) 
+ ( \tilde x_t - x )^\top( \Delta_x^t - \partial_x^t ) 
+ \eta_x \| \partial_x^t \|^2 + \eta_x \| \Delta_x^t \|^2 
\nonumber\\
&
+ \frac{1}{2\eta_y} \Big( \| y - \tilde y_t \|^2 - \| y - \tilde y_{t+1} \|^2 \Big)
+ ( y - \tilde y_t )^\top( \Delta_y^t - \partial_y^t ) 
+ \eta_y \| \partial_y^t \|^2 + \eta_y \| \Delta_y^t \|^2
\end{align}

Combining (\ref{eq:inner_combine1}) and (\ref{eq:inner_combine2}), we have
\begin{align*}
&
\hat f(x_t, y) - \hat f(x, y_t)
\leq 
\frac{\eta_x}{2} \| \Delta_x^t \|^2
+ 
\frac{\eta_y}{2} \| \Delta_y^t \|^2
\\
&
+ ( x - x_t )^\top( \Delta_x^t - \partial_x^t )
+ ( y_t - y )^\top ( \Delta_y^t - \partial_y^t )
+ \frac{\gamma}{2} \Big( \| x_t - x_0 \|^2 - \| x_{t+1} - x_0 \|^2 \Big)
\\
&
+ \Big( \frac{1}{2\eta_x} + \frac{\rho}{2} \Big) \| x - x_t \|^2 
- \Big( \frac{1}{2\eta_x} + \frac{\gamma}{2} \Big) \| x - x_{t+1} \|^2 
+ \frac{1}{2\eta_y} \Big( \| y - y_t \|^2 - \| y - y_{t+1} \|^2 \Big) 
\\
&
+ \frac{1}{2\eta_x} \Big( \| x - \tilde x_t \|^2 - \| x - \tilde x_{t+1} \|^2\Big) 
+ ( \tilde x_t - x )^\top( \Delta_x^t - \partial_x^t ) 
+ \eta_x \| \partial_x^t \|^2 + \eta_x \| \Delta_x^t \|^2 
\nonumber\\
&
+ \frac{1}{2\eta_y} \Big( \| y - \tilde y_t \|^2 - \| y - \tilde y_{t+1} \|^2 \Big)
+ ( y - \tilde y_t )^\top( \Delta_y^t - \partial_y^t ) 
+ \eta_y \| \partial_y^t \|^2 + \eta_y \| \Delta_y^t \|^2
\\
= &
  \frac{3 \eta_x}{2} \| \Delta_x^t \|^2 
+ \eta_x \| \partial_x^t \|^2
+ \frac{3 \eta_y}{2} \| \Delta_y^t \|^2
+ \eta_y \| \partial_y^t \|^2
\\
&
+ ( \tilde x_t - x_t )^\top( \Delta_x^t - \partial_x^t )
+ ( y_t - \tilde y_t )^\top ( \Delta_y^t - \partial_y^t )
+ \frac{\gamma}{2} \Big( \| x_t - x_0 \|^2 - \| x_{t+1} - x_0 \|^2 \Big)
\\
&
+ \Big( \frac{1}{2\eta_x} + \frac{\rho}{2} \Big) \| x - x_t \|^2 
- \Big( \frac{1}{2\eta_x} + \frac{\gamma}{2} \Big) \| x - x_{t+1} \|^2  
+ \frac{1}{2\eta_y} \Big( \| y - y_t \|^2 - \| y - y_{t+1} \|^2 \Big) 
\\
&
+ \frac{1}{2\eta_x} \Big( \| x - \tilde x_t \|^2 - \| x - \tilde x_{t+1} \|^2\Big) 
+ \frac{1}{2\eta_y} \Big( \| y - \tilde y_t \|^2 - \| y - \tilde y_{t+1} \|^2 \Big)
\end{align*}

Summing the above inequality from $t = 0$ to $T-1$ and using Jensen's inequality, we have
\begin{align*}
&
T \Big( \hat f(\bar x, y) - \hat f(x, \bar y) \Big)
\leq
\sum_{t=0}^{T-1} \Big( \hat f(x_t, y) - \hat f(x, y_t) \Big)
\\
\leq &
\frac{\eta_x}{2} \sum_{t=0}^{T-1} ( 3 \| \Delta_x^t \|^2 + 2 \| \partial_x^t \|^2 )
+ 
\frac{\eta_y}{2} \sum_{t=0}^{T-1} ( 3 \| \Delta_y^t \|^2 + 2 \| \partial_y^t \|^2 )
\\
&
+ \sum_{t=0}^{T-1} ( \tilde x_t - x_t )^\top( \Delta_x^t - \partial_x^t )
+ \sum_{t=0}^{T-1} ( y_t - \tilde y_t )^\top ( \Delta_y^t - \partial_y^t )
+ \frac{\gamma}{2} \Big( \| x_0 - x_0 \|^2 - \| x_T - x_0 \|^2 \Big)
\\
&
+ \Big( \frac{1}{2\eta_x} + \frac{\rho}{2} \Big) \| x - x_0 \|^2 
- \Big( \frac{1}{2\eta_x} + \frac{\gamma}{2} \Big) \| x - x_T \|^2 
+ \frac{1}{2\eta_y} \Big( \| y - y_0 \|^2 - \| y - y_T \|^2 \Big)   
\\
&
+ \frac{1}{2\eta_x} \Big( \| x - x_0 \|^2 - \| x - \tilde x_T \|^2\Big) 
+ \frac{1}{2\eta_y} \Big( \| y - y_0 \|^2 - \| y - \tilde y_T \|^2 \Big)
\end{align*}
where $\bar x = \frac{1}{T} \sum_{t=0}^{T-1} x_t$
and
$\bar y = \frac{1}{T} \sum_{t=0}^{T-1} y_t$.

Plugging in $x = \hat x(\bar y)$ and $y = \hat y(\bar x)$, we have
\begin{align*}
&
\hatGap (\bar x, \bar y)
=
\hat f(\bar x, \hat y(\bar x)) - \hat f(\hat x(\bar y), \bar y) 
\leq
\frac{1}{T} \sum_{t=0}^{T-1} \Big( \hat f(x_t, \hat y(\bar x)) - \hat f(\hat x(\bar y), y_t) \Big)
\\
\leq &
\frac{\eta_x}{2T} \sum_{t=0}^{T-1} ( 3 \| \Delta_x^t \|^2 + 2 \| \partial_x^t \|^2 )
+ 
\frac{\eta_y}{2T} \sum_{t=0}^{T-1} ( 3 \| \Delta_y^t \|^2 + 2 \| \partial_y^t \|^2 )
\\
&
+ \frac{1}{T} \sum_{t=0}^T ( \tilde x_t - x_t )^\top( \Delta_x^t - \partial_x^t )
+ \frac{1}{T} \sum_{t=0}^T ( y_t - \tilde y_t )^\top ( \Delta_y^t - \partial_y^t )
\\
&
+ \frac{1}{T} \Big( \frac{1}{\eta_x} + \frac{\rho}{2} \Big) \| \hat x(\bar y) - x_0 \|^2 
+ \frac{1}{\eta_y T} \| \hat y(\bar x) - y_0 \|^2  
\end{align*}

Taking expectation over both sides and recalling that $\E[ \| \partial_x f(x,y;\xi) \|^2 ] \leq M_1^2$ and $\E[ \| \partial_y f(x,y;\xi) \|^2 ] \leq M_2^2$, we have
\begin{align*}
&
\E [ \hatGap(\bar x, \bar y) ]
=
\E [ \hat f(\bar x, \hat y(\bar x)) - \hat f(\hat x(\bar y), \bar y) ]
\leq
\frac{1}{T} \sum_{t=0}^{T-1} \E [ \hat f(x_t, \hat y(\bar x)) - \hat f(\hat x(\bar y), y_t) ]
\\
\leq &
\frac{5 \eta_x M_1^2}{2} 
+ 
\frac{5 \eta_y M_2^2}{2} 
+ \frac{1}{T} \Big( \frac{1}{\eta_x} 
+ \frac{\rho}{2} \Big) \E [ \| \hat x(\bar y) - x_0 \|^2 ]
+ \frac{1}{\eta_y T} \E [ \| \hat y(\bar x) - y_0 \|^2 ]   .
\end{align*}

\end{proof}

\section{Proof of Lemma \ref{lemma:lhs_hat_gap}}

Before proving Lemma \ref{lemma:lhs_hat_gap}, we first state the following lemma, whose proof is in the next section.
\begin{lem}
\label{lemma:three_cases_lhs}
$\hatGap_k(\bar x_k, \bar y_k)$ could be lower bounded by the following inequalities
\begin{align}\label{eq:three_cases_lhs}
1) & ~ \hatGap_k(\bar x_k, \bar y_k) 
\geq 
( 1 - \frac{\gamma}{\rho}  ( \frac{1}{\alpha} - 1 ) ) \hatGap_{k+1}(x_0^{k+1}, y_0^{k+1})
+ \frac{\gamma}{2} ( 1 - \frac{1}{1 - \alpha} ) \| x_0^{k+1} - x_0^{k} \|^2  ,
\nonumber\\
2) & ~ \hatGap_k (\bar x_k, \bar y_k) 
\geq
P(x^{k+1}_0) - P(x^k_0) + \frac{\gamma}{2} \| \bar x_k - x_0^k \|^2 
, 
\text{ where } P(x) = \max_{y \in Y} f(x, y)  ,
\nonumber\\
3) & ~ \hatGap_k (\bar x_k, \bar y_k) 
\geq 
\frac{ \rho ( 1 - \beta ) }{ 2 (\frac{1}{\beta} - 1) } \| x_0^k - \hat x^*_k \|^2
- \frac{\rho}{2(\frac{1}{\beta} - 1)} \| \bar x_k - x_0^k \|^2   ,
\end{align}
where $0 < \alpha \leq 1$ and $0 < \beta \leq 1$.
\end{lem}

\begin{proof} (of Lemma \ref{lemma:lhs_hat_gap})
\begin{align}
\label{eq:wcc_lhs}
&
\hatGap_k(\bar x_k, \bar y_k)
\nonumber\\
= &
  \frac{1}{10} \hatGap_k(\bar x_k, \bar y_k)
+ \frac{4}{5} \hatGap_k(\bar x_k, \bar y_k)
+ \frac{1}{10} \hatGap_k(\bar x_k, \bar y_k)
\nonumber\\
\stackrel{(a)}{\geq} &
\frac{1}{10} \Big\{ \Big( 1 - \frac{\gamma}{\rho}  ( \frac{1}{\alpha} - 1 ) \Big) \hatGap_{k+1}(x_0^{k+1}, y_0^{k+1})
+ \frac{\gamma}{2} ( 1 - \frac{1}{1 - \alpha} ) \| x_0^{k+1} - x_0^{k} \|^2 ) \Big\}
\nonumber\\
&
+ \frac{4}{5} \Big\{ P(x^{k+1}_0) + \frac{\gamma}{2} \| \bar x_k - x_0^k \|^2 
- P(x_0^k) \Big\}
\nonumber\\
&
+ \frac{1}{10} \Big\{ \frac{ \rho ( 1 - \beta ) }{ 2 (\frac{1}{\beta} - 1) } \| x_0^k - \hat x^*_k \|^2
- \frac{\rho}{2(\frac{1}{\beta} - 1)} \| \bar x_k - x_0^k \|^2 \Big\}
\nonumber\\
= &
\frac{1}{10} \Big( 1 - \frac{\gamma}{\rho}  ( \frac{1}{\alpha} - 1 ) \Big) \hatGap_{k+1}( x_0^{k+1}, y_0^{k+1} )
+ \frac{4}{5} ( P(x^{k+1}_0) - P(x_0^k) )
\nonumber\\
&
+ \Big( \frac{1}{10} \cdot \frac{\gamma}{2}(1-\frac{1}{1-\alpha}) 
+ \frac{4}{5} \cdot \frac{\gamma}{2} 
- \frac{1}{10} \cdot \frac{\rho}{2(\frac{1}{\beta}-1)} \Big) \| \bar x_k - x_0^k \|^2
\nonumber\\
&
+ \frac{1}{10} ( \frac{ \rho (1-\beta) }{2 ( \frac{1}{\beta}-1 ) } ) \| x_0^k - \hat x^*_k \|^2 ]
\nonumber\\
\stackrel{(b)}{=} &
\frac{1}{10} ( 1 - 2  ( \frac{1}{\frac{5}{6}} - 1 ) ) \hatGap_{k+1}( x_0^{k+1}, y_0^{k+1} )
+ \frac{4}{5} ( P(x^{k+1}_0) - P(x_0^k) )
\nonumber\\
&
+ \Big( \frac{1}{10} \cdot \frac{\gamma}{2}(1-\frac{1}{1-\frac{5}{6}}) 
+ \frac{4}{5} \cdot \frac{\gamma}{2} 
- \frac{1}{10} \cdot \frac{\gamma}{4(\frac{1}{\frac{1}{2}}-1)} \Big) \| \bar x_k - x_0^k \|^2
\nonumber\\
&
+ \frac{1}{10} ( \frac{ \rho (1-\frac{1}{2}) }{2 ( \frac{1}{\frac{1}{2}}-1 ) } ) \| x_0^k - \hat x^*_k \|^2
\nonumber\\
= &
\frac{3}{50} \hatGap_{k+1}( x_0^{k+1}, y_0^{k+1} ) ]
+ \frac{4}{5} ( P(x^{k+1}_0) - P(x_0^k) )
\nonumber\\
&
+ \frac{\gamma}{8} \| \bar x_k - x_0^k \|^2
+ \frac{\gamma}{80} \| x_0^k - \hat x^*_k \|^2
\nonumber\\
\stackrel{(c)}{\geq} &
\frac{3}{50} \hatGap_{k+1}( x_0^{k+1}, y_0^{k+1} )
+ \frac{4}{5} ( P(x^{k+1}_0) - P(x_0^k) )
\nonumber\\
&
+ \frac{\gamma}{80} \| x_0^k - \hat x^*_k \|^2  ,
\end{align}
where inequality $(a)$ is due to Lemma \ref{lemma:three_cases_lhs},
inequality $(b)$ is due to the setting of $\gamma = 2\rho$, $\alpha = \frac{5}{6}$, $\beta = \frac{1}{2}$.
Inequality $(c)$ is due to $\| \bar x_k - x_0^k \|^2 \geq 0$.
\end{proof}

\section{Proof of Lemma~\ref{lemma:three_cases_lhs} }
\label{section_proof:lemma:three_cases_lhs}

\begin{proof}   
Before we prove the three results, we first state two results of Young's inequality as follows
\begin{align}
&
\| x - y \|^2 
=
\| x - z + z - y \|^2
=
\| x - z \|^2 + \| z - y \|^2 - 2 \langle x - z, y - z \rangle
\nonumber\\
\geq &
\| x - z \|^2 + \| z - y \|^2 - \alpha \| x - z \|^2 - \frac{1}{\alpha} \| y - z \|^2
\nonumber\\
= &
( 1 - \alpha ) \| x - z \|^2
+ ( 1 - \frac{1}{\alpha} ) \| z - y \|^2  
\nonumber\\
\label{eq:youngs1}
\Rightarrow
&
\| x - z \|^2 
\leq
\frac{1}{1 - \alpha} \| x - y \|^2 + \frac{1}{\alpha} \| y - z \|^2
\\
\label{eq:youngs2}
\Rightarrow &
- \| y - z \|^2 
\leq
- \alpha \| x - z \|^2
+ \frac{\alpha}{1 - \alpha} \| x - y \|^2  
\\
\label{eq:youngs3}
\Rightarrow &
\| x - y \|^2 
\geq
(1 - \alpha) \| x - z \|^2 + (1 - \frac{1}{\alpha}) \| y - z \|^2  ,
\end{align}
where $0 < \alpha \leq 1$.

We first consider the result 1).
\begin{align*}
&
\hatGap_{k+1} (x_0^{k+1}, y_0^{k+1})
\\
= & 
\hat f_{k+1} (x_0^{k+1}, \hat y_{k+1}(x_0^{k+1})) 
- \hat f_{k+1}(\hat x_{k+1}(y_0^{k+1}), y_0^{k+1})
\\
= &
f(x_0^{k+1}, \hat y_{k+1}(x_0^{k+1})) 
+ \frac{\gamma}{2} \| x_0^{k+1} - x_0^{k+1} \|^2
\\
&
- f(\hat x_{k+1}(y_0^{k+1}), y_0^{k+1}) 
- \frac{\gamma}{2} \| \hat x_{k+1}(y_0^{k+1}) - x_0^{k+1} \|^2
\\
= &
f(x_0^{k+1}, \hat y_{k+1}(x_0^{k+1})) 
+ \frac{\gamma}{2} \| x_0^{k+1} - x_0^{k} \|^2
- f(\hat x_{k+1}(y_0^{k+1}), y_0^{k+1}) 
- \frac{\gamma}{2} \| \hat x_{k+1}(y_0^{k+1}) - x_0^{k} \|^2
\\
&
+ \frac{\gamma}{2} \| \hat x_{k+1}(y_0^{k+1}) - x_0^{k} \|^2
- \frac{\gamma}{2} \| \hat x_{k+1}(y_0^{k+1}) - x_0^{k+1} \|^2
- \frac{\gamma}{2} \| x_0^{k+1} - x_0^{k} \|^2
\\
\stackrel{(a)}{\leq} &
f(\bar x_{k}, \hat y_{k+1}(\bar x_{k})) 
+ \frac{\gamma}{2} \| \bar x_{k} - x_0^{k} \|^2
- f(\hat x_{k+1}(\bar y_{k}), \bar y_{k}) 
- \frac{\gamma}{2} \| \hat x_{k+1}(\bar y_{k}) - x_0^{k} \|^2
\\
&
+ \frac{\gamma}{2} \Big\{ \frac{1}{\alpha} \| \hat x_{k+1}(y_0^{k+1}) - x_0^{k+1} \|^2 
                  + \frac{1}{1 - \alpha} \| x_0^{k+1} - x_0^{k} \|^2 \Big\}
\\
&
- \frac{\gamma}{2} \| \hat x_{k+1}(y_0^{k+1}) - x_0^{k+1} \|^2
- \frac{\gamma}{2} \| x_0^{k+1} - x_0^{k} \|^2
\\
= &
\hat f_k(\bar x_{k}, \hat y_k(\bar x_{k}))
- \hat f_k(\hat x_{k+1}(\bar y_{k}), \bar y_{k})
\\
&
+ \frac{\gamma}{2} ( \frac{1}{\alpha} - 1 ) \| \hat x_{k+1}(y_0^{k+1}) - x_0^{k+1} \|^2 
+ \frac{\gamma}{2} ( \frac{1}{1 - \alpha} - 1 ) \| x_0^{k+1} - x_0^{k} \|^2 
\\
\stackrel{(b)}{\leq} &
\hat f_k(\bar x_{k}, \hat y_k(\bar x_{k}))
- \hat f_k(\hat x_{k+1}(\bar y_{k}), \bar y_{k})
\\
&
+ \frac{\gamma}{2} ( \frac{1}{\alpha} - 1 ) \frac{2}{\rho} ( \hat f_{k+1} (x_0^{k+1}, y_0^{k+1}) - \hat f_{k+1} (\hat x_{k+1}(y_0^{k+1}), y_0^{k+1}) )
+ \frac{\gamma}{2} ( \frac{1}{1 - \alpha} - 1 ) \| x_0^{k+1} - x_0^{k} \|^2 
\\
\stackrel{(c)}{\leq} &
\hatGap_k (\bar x_{k}, \bar y_{k})
+ \frac{\gamma}{\rho}  ( \frac{1}{\alpha} - 1 ) \hatGap_{k+1} (x_0^{k+1}, y_0^{k+1}) 
+ \frac{\gamma}{2} ( \frac{1}{1 - \alpha} - 1 ) \| x_0^{k+1} - x_0^{k} \|^2   ,
\end{align*}
where inequality $(a)$ is due to (\ref{eq:youngs1}) ($0 < \alpha \leq 1$).
Inequality $(b)$ is due to $\rho$-strong convexity of $\hat f_{k+1}(x, y_0^{k+1})$ in $x$ and optimality at $\hat x_{k+1}(y_0^{k+1})$.
Inequality $(c)$ is due to $\hat f_k(\hat x_{k+1}(\bar y_{k}), \bar y_{k}) \geq \hat f_k(\hat x_{k}(\bar y_{k}), \bar y_{k})$ 
and 
$\hat f_{k+1}(x_0^{k+1}, y_0^{k+1}) \leq \hat f_{k+1}(x_0^{k+1}, \hat y_{k+1} (x_0^{k+1}))$.

Re-organizing the above inequality, we have
\begin{align*}
\hatGap_k (\bar x_{k}, \bar y_{k}) 
\geq &
( 1 - \frac{\gamma}{\rho} ( \frac{1}{\alpha} - 1 ) ) \hatGap_{k+1}(x_0^{k+1}, y_0^{k+1})
+ \frac{\gamma}{2} ( 1 - \frac{1}{1 - \alpha} ) \| x_0^{k+1} - x_0^{k} \|^2   ,
\end{align*}
which proves result $1)$.

Then we turn to result 2) as follows.
\begin{align*}
\hatGap_k (\bar x_k, \bar y_k) 
= &
\hat f_k(\bar x_k, \hat y_k(\bar x_k)) - \hat f_k(\hat x_k(\bar y_k), \bar y_k)
\\
\geq &
\hat f_k(\bar x_k, \hat y_k(\bar x_k)) - \hat f_k(x_0^k, \bar y_k)
\\
\geq &
\hat f_k(\bar x_k, \hat y_k(\bar x_k)) - \hat f_k(x_0^k, \hat y_k(x_0^k))
\\
= &
f(\bar x_k, \hat y_k(\bar x_k)) + \frac{\gamma}{2} \| \bar x_k - x_0^k \|^2 
- f(x_0^k, \hat y_k(x_0^k)) - 0
\\
= & 
f(x^{k+1}_0, \hat y_k(x_0^{k+1})) + \frac{\gamma}{2} \| \bar x_k - x_0^k \|^2 
- f(x_0^k, \hat y_k(x_0^k))  
\\
= &
P(x^{k+1}_0) - P(x^k_0) + \frac{\gamma}{2} \| \bar x_k - x_0^k \|^2    ,
\end{align*}
which proves result $2)$.

Result 3) can be proved as follows
\begin{align*}
\| \bar x_k - x_0^k \|^2
\stackrel{(a)}{\geq} &
(1 - \beta) \| x_0^k - \hat x^*_k \|^2 + ( 1 - \frac{1}{\beta} ) \| \hat x^*_k - \bar x_k \|^2
\\
\stackrel{(b)}{\geq} &
(1 - \beta) \| x_0^k - \hat x^*_k \|^2 
+ ( 1 - \frac{1}{\beta} ) \frac{2}{\rho} ( \hat f_k(\bar x_k, \hat y^*_k) - \hat f_k(\hat x^*_k, \hat y^*_k) )
\\
\stackrel{(c)}{\geq} &
(1 - \beta) \| x_0^k - \hat x^*_k \|^2 
+ ( 1 - \frac{1}{\beta} ) \frac{2}{\rho} \hatGap_k(\bar x_k, \bar y_k) )
\\
\Rightarrow 
\hatGap_k(\bar x_k, \bar y_k)
\geq &
\frac{ \rho ( 1 - \beta ) }{ 2 (\frac{1}{\beta} - 1) } \| x_0^k - \hat x^*_k \|^2
- \frac{\rho}{2(\frac{1}{\beta} - 1)} \| \bar x_k - x_0^k \|^2  ,
\end{align*}
where inequality $(a)$ is due to (\ref{eq:youngs3}) and $0 < \beta \leq 1$.
Inequality $(b)$ is due to $\rho$-storng convexity of $\hat f_k$ in $x$.
Inequality $(c)$ is due to $0 < \beta \leq 1$ and
\begin{align*}
\hat f_k(\bar x_k, \hat y^*_k) - \hat f_k(\hat x^*_k, \hat y^*_k) 
\leq &
\hat f_k(\bar x_k, \hat y_k(\bar x_k)) - \hat f_k(\hat x^*_k, \bar y_k) 
\\
\leq &
\hat f_k(\bar x_k, \hat y_k(\bar x_k)) - \hat f_k(\hat x_k (\bar y_k), \bar y_k) 
\\
= &
\hatGap_k (\bar x_k, \bar y_k)  .
\end{align*}
\end{proof}  

\end{document}